\documentclass[leqno]{article}
\usepackage{geometry}
\usepackage{graphicx}
\usepackage[cp1251]{inputenc}
\usepackage[T2A]{fontenc}
\usepackage[english]{babel}
\usepackage{mathtools}
\usepackage{amsfonts,amssymb,mathrsfs,amscd,amsmath,amsthm}
\usepackage{enumitem}

\usepackage[usenames,dvipsnames,svgnames,table]{xcolor}

\usepackage{verbatim}

\usepackage{url}

\ifpdf
\usepackage{hyperref}
\else
\usepackage[dvipdfmx]{hyperref}
\fi

\def\thtext#1{
\catcode`@=11
\gdef\@thmcountersep{. #1}
\catcode`@=12
}

\def\threst{
\catcode`@=11
\gdef\@thmcountersep{.}
\catcode`@=12
}

\theoremstyle{plain}

\newtheorem{thm}{Theorem}[section]
\newtheorem{prop}[thm]{Proposition}
\newtheorem{cor}[thm]{Corollary}

\newtheorem{lem}[thm]{Lemma}

\theoremstyle{definition}

\newtheorem{conj}[thm]{Conjecture}
\newtheorem{rk}[thm]{Remark}

\newtheorem{examp}[thm]{Example}

\catcode`@=11
\def\.{.\spacefactor\@m}
\catcode`@=12

\def\:{\colon}
\def\0{\emptyset}
\def\<{\langle}
\def\>{\rangle}
\def\imply{ $\Rightarrow$ }
\def\rom#1{{\rm #1}}
\def\({\rom(}
\def\){\rom)}

\def\d{\partial}
\def\sm{\setminus}
\def\ss{\subset}
\def\sp{\supset}
\def\x{\times}
\def\tolim{\operatornamewithlimits{\longrightarrow}\limits}
\def\torllim{\operatornamewithlimits{\rightleftarrows}\limits}

\def\cA{\mathcal{A}}
\def\cH{\mathcal{H}}
\def\cK{\mathcal{K}}
\def\cCK{\mathcal{CK}}
\def\cM{\mathcal{M}}
\def\cP{\mathcal{P}}
\def\cR{\mathcal{R}}
\def\cU{\mathcal{U}}
\def\cV{\mathcal{V}}

\def\N{\mathbb{N}}
\def\R{\mathbb{R}}

\def\a{\alpha}
\def\b{\beta}
\def\dl{\delta}
\def\D{{\Delta}}
\def\e{\varepsilon}
\def\g{\gamma}
\def\l{\lambda}
\def\r{\rho}
\def\s{\sigma}
\def\t{\tau}
\def\v{\varphi}

\def\CARD{\operatorname{CARD}}
\def\codis{\operatorname{codis}}
\def\cov{\operatorname{cov}}
\def\diam{\operatorname{diam}}
\def\dis{\operatorname{dis}}
\def\dist{\operatorname{dist}}
\def\GH{\operatorname{\mathcal{G\!H}}}
\def\Id{\operatorname{Id}}
\def\ind{\operatorname{ind}}
\def\Ind{\operatorname{Ind}}
\def\ord{\operatorname{ord}}
\def\ORD{\operatorname{ORD}}
\def\TOP{\operatorname{TOP}}

\def\gA{\mathfrak{A}}
\def\gB{\mathfrak{B}}
\def\gC{\mathfrak{C}}

\begin{document}

\author{Semeon A.~Bogaty, Alexey A.~Tuzhilin}
\title{Fundamentals of Theory of Continuous Gromov--Hausdorff distance.}
\date{}
\maketitle

\tableofcontents

\begin{abstract}
The Gromov--Hausdorff distance (hereinafter referred to as the GH-distance) is a measure of non-isometricity of metric spaces. In this paper, we study a modification of this distance that also takes topological differences into account. The resulting function of pairs of metric spaces is called the continuous GH-distance. We show that many basic properties of the classical GH-distance also hold in the continuous case. However, the continuous GH-distance, distinguishing between topologies, can differ significantly from the classical one. We will provide numerous examples of this distinction and demonstrate the role of topological dimension here. In particular, we will prove that the continuous GH-distance, like the classical one, is intrinsic, but, unlike the classical one, it is incomplete. Since we are dealing with all metric spaces, we will show, within the framework of the von Neumann-Bernays-G\"odel set theory, how topological concepts can be transferred to proper classes. \\

\end{abstract}

\section{Introduction.}

This paper develops the theory of a variation on Gromov--Hausdorff distance geometry. Classical Gromov--Hausdorff distance measures the difference between two metric spaces: if the spaces are isometric, then they are at zero distance from each other, and the greater the distance, the greater their metric differences. The classical definition of the Gromov-Hausdorff distance does not take into account additional structures that metric spaces may possess. Even the topology generated by the metric is ignored by this distance. This feature of the Gromov-Hausdorff distance was modified, for example, by Rieffel~\cite{Rieffel}, who proposed a more subtle comparison of the so-called quantum metric spaces. In the paper~cite{LimMemoliSmith}, another modification of the Gromov-Hausdorff distance was proposed that takes continuity into account. It was noted that comparing spheres in Euclidean space, endowed with the standard intrinsic metric, the classical Gromov-Hausdorff distance differs from those obtained with its continuous analog. Another version of the continuous Gromov--Hausdorff distance was proposed by~\cite{LeeMorales} for comparing dynamical systems and solutions of partial differential equations. However, these authors' approach significantly complicates the technique, since their version of the distance does not satisfy the triangle inequality. We decided to use the definition from~\cite{LimMemoliSmith} and begin a more fundamental study of the continuous Gromov--Hausdorff distance. We do not restrict ourselves to compact metric spaces, which constitute the traditional domain of the classical Gromov--Hausdorff distance, but consider all spaces, which also leads us to some difficulties associated with the paradoxes of set theory. For the sake of completeness, we present in the supplementary Section~\ref{sec:add} a number of technical results that make it possible to work with such ``monsters'' as proper classes in the sense of von Neumann--Bernays--G\"odel~\cite{Mendelson,TBanach}.

To formulate the main results of this paper, we begin with a reminder of the necessary definitions. Let $(X,\rho)$ be a metric space, and $x$, $y$ be its points, then $|xy|=\r(x,y)$ will denote the distance between these points, and if $A$ and $B$ are nonempty subsets of $X$, then we put $|AB|=|BA|=\inf\bigl\{|ab|:a\in A,\,b\in B\bigr\}$. If $A=\{a\}$, then instead of $\bigl|\{a\}B\bigr|=\bigl|B\{a\}\bigr|$ we will write $|aB|=|Ba|$. For a positive real $s$ and a non-negative real $r$ we define
\begin{itemize}
\item the \emph{open ball with center $a\in X$ and radius $s$} as
$$
U_s(a)=\{x\in X:|xa|<s\};
$$
\item the \emph{open $s$-neighborhood of non-empty $A\ss X$} as
$$
U_s(A)=\{x\in X:|xA|<s\};
$$
\item the \emph{closed ball with center $a\in X$ and radius $r$} as
$$
B_r(a)=\{x\in X:|xa|\le r\};
$$
\item the \emph{closed $r$-neighborhood of non-empty $A\ss X$} as
$$
B_r(A)=\{x\in X:|xA|\le r\}.
$$
\end{itemize}
For non-empty subsets $A,B\ss X$, the value
$$
d_H(A,B)=\max\bigl\{\sup_{a\in A}|aB|,\,\sup_{b\in B}|Ab|\bigr\}
$$
is called the \emph{Hausdorff distance between $A$ and $B$}. An equivalent definition is:
$$
d_H(A,B)=\inf\bigl\{r:A\ss B_r(B)\ \ \text{and}\ \ B\ss B_r(A)\bigr\}.
$$
It is well known that $d_H$ is a generalized pseudometric on the set of all non-empty subsets of a metric space $X$. Here, the word ``generalized'' means that $d_H$ may be equal to infinity on some pairs of subsets (for example, on a bounded and an unbounded set), and the word ``pseudometric'' means that $d_H$ may be equal to zero on some pairs of distinct subsets (for example, on a non-closed set and its closure).

M. Gromov defined the distance between two non-empty metric spaces~\cite{Gromov1981, Gromov1999, BurBurIva01}. \emph{The Gromov--Hausdorff distance $d_{GH}(X,Y)$} is the greatest lower bound of the Hausdorff distances between the images of the spaces $X$ and $Y$ under their isometric embeddings in all possible metric spaces.

In this paper, we study a continuous analogue of the Gromov--Hausdorff distance $d_{GH}^c(X,Y)$ (Formula~(\ref{eq:7})). We show that
\begin{itemize}
\item the continuous Gromov--Hausdorff distance $d_{GH}^c(X,Y)$ is a generalized pseudometric\footnote{In~\cite{LeeMorales}, some analogue of the continuous Gromov--Hausdorff distance is also considered, but there this distance does not satisfy the triangle inequality, which leads to numerous technical difficulties. Instead of that, we follow the approach of~\cite{LimMemoliSmith}.} (Proposition~\ref{prop:triangle}) and it dominates the Gromov--Hausdorff metric $d_{GH}(X,Y)\le d_{GH}^c(X,Y)$ (Property (\ref{prop:GHelementProps:2}) of Proposition~\ref{prop:GHelementProps});
\item for zero-dimensional metric spaces these distances coincide (Corollary~\ref{cor:dim0}), and the continuous distance from a connected space $X$ to any zero-dimensional space is not less than the diameter of this connected space (Proposition~\ref{prop:ConnectDiscrete});
\item although every continuum in the Euclidean space $\R^n$, $n\ge 2$, is at zero distance from the set of all pseudo-arcs (Bing's Theorem~\cite{B}), its continuous Gromov--Hausdorff distance to the set of all pseudo-arcs is equal to half the diameter of this continuum (Proposition~\ref{prop:LinConnAndContinuum});
\item Cook's continuum provides a contrasting example of the difference between the Hausdorff and Gromov--Hausdorff metrics on one hand and the continuous Gromov--Hausdorff metric on the other (Corollary~\ref{cor:GHCont-intrinsic});
\item a metric space $Y$ at zero continuous Gromov--Hausdorff distance from a sphere $S^n$, $n\ge 1$, with an arbitrary metric, is isometric to this sphere (Corollary~\ref{cor:sphere}). But for every metric on the sphere $S^n$, $n\ge 1$, there are at least continuum pairwise non-isometric metric spaces at zero classical Gromov--Hausdorff distance from this sphere $S^n$ (for example, obtained from the sphere by discarding of finite sets of points that are pairwise nonisometric to each other).
\end{itemize}

We also show that, unlike the classical Gromov--Hausdorff distance~\cite{BT21, BT23, BBRT25}, the continuous Gromov--Hausdorff distance is not complete (Theorem~\ref{thm:noncomplete}). Here it should be emphasized that all metric spaces do not form a set, and therefore do not form a topological space. For complete rigor of exposition, a transition to the theory of von Neumann--Bernays--G\"odel classes~\cite{Mendelson} is necessary. Taking this into account, the concept of completeness of the (continuous) Gromov--Hausdorff distance is well defined~\cite{BT21}. The assertion we proved that the continuous Gromov--Hausdorff distance is intrinsic (Theorem~\ref{thm:GHCont-intrinsic}) also has a correct meaning. However, in order to remain within the framework of set theory, one can only consider metric spaces of some cardinalities bounded from above by some fixed one.

The work of A.A. Tuzhilin was supported by grant No. 25-21-00152 of the Russian Science Foundation, by National Key R\&D Program of China (Grant No. 2020YFE0204200), as well as by the Sino-Russian Mathematical Center at Peking University. Partial of the work by A.A. Tuzhilin were done in Sino-Russian Math. center, and he thanks the Math. Center for the invitation and the hospitality.

\section{Preliminaries}
The proper class of all metric spaces considered up to isometry is denoted by $\GH$, and the subset of $\GH$ consisting of all compact metric spaces is denoted by $\cM$.

A more technical definition of the Gromov--Hausdorff distance is often used. For non-empty metric spaces $X$ and $Y$, the set of non-empty relations between $X$ and $Y$, i.e., non-empty subsets of the Cartesian product $X\times Y$, is denoted by $\cP_0(X\times Y)$. For $\s\in\cP_0(X\times Y)$, we define the \emph{distortion\/} by setting
\begin{equation}\label{eq:1}
\dis\s=\sup\Bigl\{\bigl||xx'|-|yy'|\bigr|:(x,y),\,(x',y')\in\s\Bigr\}.
\end{equation}
In particular, if $f\:X\to Y$ is an arbitrary mapping, then we define the \emph{distortion $\dis f$} for it as the distortion of its graph
\begin{equation}\label{eq:2}
\dis f=\sup\Bigl\{\bigl||xx'|-|f(x)f(x')|\bigr|:x,x'\in X \Bigr\}.
\end{equation}

Let us note a well-known property of distortion of the relations composition.

\begin{prop}\label{prop:DisCompose}
If $\s\in\cP_0(X\times Y)$, $\t\in\cP_0(Y\times Z)$ and $\t\circ\s\ne\0$ then $\dis(\t\circ\s)\le\dis\s+\dis\t$.
\end{prop}

A multivalued surjective mapping $R$ from $X$ to $Y$ is called a \emph{correspondence between $X$ and $Y$}. The set of all correspondences between $X$ and $Y$ is denoted by $\cR(X,Y)$. The Gromov--Hausdorff distance is calculated by the formula~\cite[Theorem 7.3.25]{BurBurIva01}
\begin{equation}\label{eq:3}
d_{GH}(X,Y)=\frac12\inf\bigl\{\dis R:R\in\cR(X,Y)\bigr\}.
\end{equation}

A special case of a correspondence can be constructed for any pair of mappings $f\:X\to Y$ and $g\:Y\to X$ by setting $R_{f,g}=f\cup g^{-1}$, where $g^{-1}$ denotes the inverse relation to the relation $g$. On the other hand, in each correspondence $R\in\cR(X,Y)$, a subcorrespondence $R_{f,g}$ can be extracted if $f\:X\to Y$ and $g\:Y\to X$ are chosen such that $f\ss R$ and $g\ss R^{-1}$. It is easy to see that the distortion of each subcorrespondence does not exceed the distortion of the correspondence, in particular, $\dis R_{f,g}\le\dis R$. To calculate the distortion of the correspondence $R_{f,g}$, we need the \emph{codistortion $\codis(f,g)$}, defined as follows:
\begin{equation}\label{eq:4}
\codis(f,g)=\sup\Bigl\{\bigr||x\,g(y)|-|f(x)\,y|\bigr|:x\in X,\,y\in Y\,\Bigr\}.
\end{equation}

\begin{prop}\label{prop:CorrespInTermsFuncs}
For any $X,Y\in\GH$, a correspondence $R\in\cR(X,Y)$ and any two mappings $f\:X\to Y$ and $g\:Y\to X$ such that $f\ss R$ and $g\ss R^{-1}$, the following holds\/\rom:
\begin{equation}\label{eq:5}
\dis R_{f,g}=\max\bigl\{\dis f,\,\dis g,\,\codis(f,g)\bigr\}\le \dis R.
\end{equation}
\end{prop}

From Equality (\ref{eq:3}) and Proposition~\ref{prop:CorrespInTermsFuncs} the following result is immediately obtained.

\begin{cor}[see {\cite[p.~3]{LimMemoliSmith}}]\label{cor:GH-metri-and-functions}
For any $X,Y\in \GH$, the following holds\/\rom:
\begin{equation}\label{eq:6}
d_{GH}(X,Y)=\frac12\,\inf_{\substack{f\:X\to Y\\ g\:Y\to X}}\,\max\bigl\{\dis f,\,\dis g,\,\codis (f,g)\bigr\}.
\end{equation}
\end{cor}

It follows from Proposition~\ref{prop:DisCompose} that the distortion of the composition of mappings does not exceed the sum of the distortions of these mappings.

\begin{prop}
Let three metric spaces $X$, $Y$ and $Z$ be given, as well as mappings $X\torllim^f_gY\torllim^h_kZ$.
Then
$$
\codis(h\circ f,g\circ k)\le\codis(f,g)+\codis(h,k).
$$
\end{prop}

\begin{proof}
Let's choose arbitrary $x\in X$, $z\in Z$ and let's put $y_1=f(x)$, $y_2=k(z)$. Then
\begin{multline*}
\Bigl|\bigl|x\,(g\circ k)(z)\bigr|-\bigl|(h\circ f)(x)\,z\bigr|\Bigr|=
\Bigl|\bigl|x\,g(y_2)\bigr|-\bigl|h(y_1)\,z\bigr|\Bigr|=\\ =
\Bigl|\bigl|x\,g(y_2)\bigr|-\bigl|y_1y_2\bigr|+\bigl|y_1y_2\bigr|-\bigl|h(y_1)\,z\bigr|\Bigr|\le \\
\le \Bigl|\bigl|x\,g(y_2)\bigr|-\bigl|f(x)y_2\bigr|\Bigr|+
\Bigl|\bigl|y_1k(z)\bigr|-\bigl|h(y_1)\,z\bigr|\Bigr|\le\codis(f,g)+\codis(h,k).
\end{multline*}
The result follows from the arbitrariness of $x$ and $z$.
\end{proof}

If in Corollary~\ref{cor:GH-metri-and-functions} we restrict ourselves to continuous mappings $f$ and $g$, then we obtain the definition of the \emph{continuous Gromov--Hausdorff distance\/}:
\begin{equation}\label{eq:7}
d^c_{GH}(X,Y)=\frac12\,\inf_{\substack{f\in C(X,Y)\\ g\in C(Y,X)}}\,\dis R_{f,g}=
\frac12\,\inf_{\substack{f\in C(X,Y)\\ g\in C(Y,X)}}\,\max\bigl\{\dis f,\,\dis g,\,\codis (f,g)\bigr\},
\end{equation}
where $C(Z,W)$ denotes the set of all continuous mappings from the metric space $Z$ to the metric space $W$.
For brevity, we will also call this distance the \emph{continuous GH-distance}.

For a (non-single-point) metric space $X$ and a point $x\in X$, we define the circumscribed and inscribed radii (centered at $x$) as $R_x=\sup\bigl\{|xx'|:x'\in X\bigr\}$ and $r_x=\inf\bigl\{|xx'|:x'\in X,\ x'\ne x\bigr\}$, respectively. We introduce four values:
\begin{align*}
\diam(X)&=\sup\bigl\{R_x:x\in X\bigr\}=\sup\bigl\{|xx'|:x,x'\in X\bigr\}\ \text{(\emph{diameter\/})}, \\
R(X)&=\inf\bigl\{R_x:x\in X\bigr\}\ \text{(\emph{Chebyshev radius\/})}, \\
d(X)&=\sup\bigl\{r_x:x\in X\bigr\},\\
s(X)&=\inf\bigl\{r_x:x\in X\bigr\}=\inf\bigl\{|xx'|:x,x'\in X,\ x'\ne x\bigr\}.
\end{align*}

For a one-point metric space we reserve the notation $\D_1$. It is natural to assume $\diam(\D_1)=R(\D_1)=d(\D_1)=s(\D_1)=0$. It is known~\cite[p\. 255]{BurBurIva01} that $2d_{GH}(\D_1,X)=\diam X$. Note that for an unbounded space, it holds $\diam(X)=R(X)=\infty$.

It is clear that for a non-single-point metric space, $d(X)=0$ if and only if $X$ has no isolated points. For a non-single-point finite metric space, $s(X)>0$. When $s(X)>0$ we say that the metric is bounded away from zero. For the standard sphere $S^m\ss\R^{m+1}$, equipped with both the metric induced from $\R^{m+1}$ and the intrinsic metric, we have $0=s(S^m)=d(S^m)<R(S^m)=\diam(S^m)$; for a ball $B^m\ss\R^m$, we have $0=s(B^m)=d(B^m)<2R(B^m)=\diam(B^m)$.

The next result is obvious.

\begin{prop}\label{prop:inequv}
For an arbitrary metric space $X$, the following chains of inequalities are valid
$$
0\le s(X)\le d(X)\le \diam(X)\le 2R(X)\ \ \text{and}\ \ s(X)\le R(X)\le \diam(X).
$$
\end{prop}

For a metric space $(X,\r)$ and a number $\l\ge 0$, we will denote by $\l X$ the set $X$ with (pseudo)met\-ric $\l\r$.

\subsection{Dimensions of Topological Spaces}
We will present three different notions of dimension, which coincide in the case of separable metric spaces and may differ for more general metric and topological spaces. We will begin with dimension, which in different sources is called either the Lebesgue dimension, or the covering dimension, or the topological dimension. This is the one we will primarily use. We will give the classical definition from~\cite{E85}. Recall that for a covering $\cU$ of a topological space $X$, its \emph{multiplicity $\ord\,\cU$} is defined to be either the smallest natural number $n$ such that every point of $X$ is contained in at most $n$ elements of $\cU$, or, if such $n$ does not exist, then the infinity symbol $\infty$. A family $\cA$ of subsets of a topological space $X$ is called \emph{locally finite\/} if every point of $X$ has a neighborhood that intersects only finitely many elements of $\cA$. An important special case of such families are \emph{locally finite coverings of $\cU$}.

\begin{thm}[{\cite[1.1.12]{E85}}]
Let $\mathcal{F}$ be a locally finite family of closed sets of a topological space $X$. Then the union of elements of this family is a closed subset of $X$.
\end{thm}

Furthermore, a covering $\cV$ of $X$ is said to be \emph{inscribed\/} in a covering $\cU$ of this space (we denote this property by $\cV\succ\cU$) if for every $V\in\cV$ there is $U\in\cU$ such that $V\ss U$. Of most interest to us are open coverings of a topological space $X$, so we denote the set of such coverings by $\cov(X)$. The subfamily of $\cov(X)$ consisting of finite coverings is denoted by $\cov_f(X)$.

\begin{thm}[Stone, see also~{\cite[4.4.1]{E85}}]\label{thm:Stone}
Every open covering of a metric space can be inscribed with a locally finite open covering.
\end{thm}

A topological space is called \emph{paracompact\/} if every open covering of it can be inscribed with a locally finite covering. In these terms, Theorem~\ref{thm:Stone} states: \emph{every metric space is paracompact}.

Note that each of the three dimensions is defined for its own class of metric spaces (regular, normal, Tychonoff). In what follows, we will apply the corresponding results exclusively to metric spaces, so we will assume from the outset that the space $X$ for which we will define dimensions is Hausdorff and normal, like every metric space: such an $X$ is always automatically regular and Tychonoff.

We first define the \emph{Lebesgue dimension $\dim\,X$}, also called the \emph{covering dimension\/} or the \emph{topological dimension\/}:
\begin{itemize}
\item if $X=\0$, then $\dim\,X=-1$;
\item if $X\ne\0$, then we set
$$
\dim\,X=-1+\sup_{\cU\in\cov_f(X)}\,\inf_{\substack{\cV\in\cov(X)\\ \cV\succ\cU}}\ord\,\cV.
$$
\end{itemize}
In other words, for a non-empty space $X$, we consider $n\in\mathbb{N}\cup\{\infty\}$ such that every finite open covering $\cU$ of $X$ can be inscribed with an open covering $\cV$ of multiplicity at most $n+1$, and we take the smallest of these $n$.

Note that $\dim\,X=0$ is equivalent to the possibility of inscribing an open partition into every finite open covering. Spaces $X$ for which $\dim\,X=0$ are called \emph{zero-dimensional}.

\begin{examp}\label{examp:DimZeroDiscrete}
Let $X$ be a nonempty discrete topological space. Then every finite open covering of $X$ can be inscribed with a covering of singleton open sets, hence $\dim\,X=0$.
\end{examp}

\begin{rk}
In a number of monographs, for example in {\rm\cite{Munkres}}, when defining $\dim\,X$, any, not necessarily finite, coverings of $\cU$ are considered. We will follow a more traditional definition.
\end{rk}

\begin{thm}[Dowker~\cite{Dowker}, see also~{\cite[7.2.4]{E85}}]\label{thm:Dowker}
A normal Hausdorff space $X$ satisfies $\dim\,X\le n$ if and only if each of its locally finite open coverings can be inscribed with a covering of multiplicity at most $n$.
\end{thm}

We now define two \emph{inductive dimensions\/} of a normal Hausdorff space $X$: \emph{small $\ind\,X$} and \emph{large $\Ind\,X$}. Recall that every subspace of a Hausdorff space is also Hausdorff, and for normal spaces, a more subtle assertion holds: every closed subspace of a normal space is itself normal. In particular, since the boundary of an arbitrary subset of a topological space is closed, the boundaries of normal Hausdorff topological spaces are also closed. We will use this idea in defining the inductive dimensions.

So, let
\begin{itemize}
\item $\ind\,X=-1$ precisely when $X=\0$;
\item if $X\ne\0$ and for $n\in\mathbb{N}\cup\{0\}$ we have already defined what $\ind\,Y\le n-1$ means for Hausdorff normal topological spaces, then we put $\ind\,X\le n$ if and only if for every point $x\in X$ and any of its neighborhoods $V$ there exists an open neighborhood $U$, $x\in U\ss V$, satisfying $\ind\,\d U\le n-1$;
\item put $\ind\,X=n$ if and only if $\ind\,X\le n$ and it is not true that $\ind\,X\le n-1$;
\item if the number $n$ described above does not exist, then we put $\ind\,X=\infty$.
\end{itemize}

\begin{examp}\label{examp:indVanishes}
For a nonempty $X$, the condition $\ind\,X=0$ means that every point $x\in X$ in each of its neighborhoods $V$ has an open neighborhood $U$ with empty boundary, i.e., $U$ is a clopen subset of $X$. If such a $U$ is distinct from $X$, then $X\sm U$ is also a nonempty open subset of $X$. Thus, the connected components of $X$ are only singleton subsets. Indeed, if $Y\ss X$ consists of more than one point and $x\in Y$, then for every point $y\in Y$ distinct from $x$ there are, by the Hausdorff property, disjoint neighborhoods $V^x$ and $V^y$. By choosing in $V^x$ an open neighborhood $U$ of $x$ in $X$, we partition the set $Y$ into two nonempty open sets $Y\cap U$ and $Y\cap(X\sm U)$, so that $Y$ is disconnected. Recall that a topological space in which all connected components are one-point subsets is called \emph{totally disconnected}. Thus, we have shown that the condition $\ind\,X=0$ implies the total disconnectedness of $X$.
\end{examp}

\begin{examp}\label{examp:indNotVanishe}
Now let us extract from Example~\ref{examp:indVanishes} a consequence for the space $X$ with $\ind\,X>0$. The last condition means that for some point $x\in X$, the zero-dimensionality requirement is not satisfied. This means that there exists a neighborhood $V$ of $x$ for which every neighborhood $U\ss V$ of this point is not clopen. Thus, every clopen neighborhood of $x$ necessarily intersects $F:=X\sm V$. Thus, we get the following formulation: \textbf{the condition $\ind\,X>0$ is equivalent to the existence of a point $x\in X$ and a closed set $F$ not containing it such that every clopen neighborhood of $x$ intersects $F$}.
\end{examp}

We now define the large inductive dimension. Let
\begin{itemize}
\item $\Ind\,X=-1$ precisely when $X=\0$;
\item if $X\ne\0$ and for $n\in\mathbb{N}\cup\{0\}$ we have already defined what $\Ind\,Y\le n-1$ means for Hausdorff normal topological spaces, then we put $\Ind\,X\le n$ if and only if for every closed $F\ss X$ and every open $V\ss X$ such that $F\ss V$, there exists an open $U\ss X$, $F\ss U\ss V$, satisfying $\Ind\,\d U\le n-1$;
\item put $\Ind\,X=n$ if and only if $\Ind\,X\le n$ and it is not true that $\Ind\,X\le n-1$;
\item if the number $n$ described above does not exist, then we set $\Ind\,X=\infty$.
\end{itemize}

\begin{examp}\label{examp:IndVanishesNotVanish}
We will carry out reasoning similar to that in Examples~\ref{examp:indVanishes} and~\ref{examp:indNotVanishe}. For a nonempty $X$, the condition $\Ind\,X=0$ means that for every closed set $F\ss X$, every open $V\sp F$ contains an open $U\sp F$ with empty boundary, i.e., $U$ is a clopen subset of $X$. In particular, \emph{if $X$ is a nonempty discrete space, then $\Ind\,X=0$.}

Let us now formulate an equivalent reformulation: the condition $\Ind\,X>0$ means the existence of a closed set $F\ss X$ and an open set $V\sp F$ such that each clopen set $U\sp F$ intersects the closed set $H:=X\sm V$. Thus, we arrive at the following variant: \emph{the condition $\Ind\,X>0$ is equivalent to the existence of disjoint closed sets $F,H\ss X$ such that every clopen set $U\sp F$ intersects $H$.} If in the previous statement we take $X$ to be a metric space and replace the condition $F\cap H=\0$ with $|FH|>0$, then we obtain the Nagami--Roberts theorem~\cite{NR65}: \emph{for a metric space $X$, the condition $\Ind\,X>0$ is equivalent to the existence of closed sets $F,H\ss X$ such that $|FH|>0$ and every clopen set $U\sp F$ intersects $H$}.
\end{examp}

\begin{rk}
Since all singleton subsets of $X$ are closed by virtue of the Hausdorff property, $\ind\,X\le\Ind\,X$.
\end{rk}

\begin{thm}[\cite{E85}]\label{thm:SeparMetriAllDimsEq}
For any separable metric space $X$, it holds $\dim\,X=\ind\,X=\Ind\,X$.
\end{thm}

Recall that a topological space is called \emph{Lindel\"of\/} if any of its open covering contains an at most countable subcovering. Examples of Lindel\"of spaces include both compacta and spaces with a countable base.

\begin{thm}[\cite{E85}]
For any Lindel\"of $X$, the conditions $\dim\,X=0$, $\ind\,X=0$, and $\Ind\,X=0$ are equivalent.
\end{thm}

\section{General Properties of Continuous Gromov--Hausdorff Distance}

\begin{prop}\label{prop:triangle}
The distance $d^c_{GH}$ satisfies the triangle inequality.
\end{prop}

\begin{proof}
Let $X$, $Y$, and $Z$ be arbitrary metric spaces. We will show that
$$
d^c_{GH}(X,Z)\le d^c_{GH}(X,Y)+d^c_{GH}(Y,Z).
$$
If one of the distances $d^c_{GH}(X,Y)$ or $d^c_{GH}(Y,Z)$ is infinite, then the inequality holds. Now let both of these distances be finite. For any $\e > 0$, choose $f\in C(X,Y)$ and $g\in C(Y,X)$, as well as $h\in C(Y,Z)$ and $k\in C(Z,Y)$, for which
$d^c_{GH}(X,Y)\ge\max\bigl\{\dis f,\,\dis g,\,\codis (f,g)\bigr\}-\e$ and
$d^c_{GH}(Y,Z)\ge\max\bigl\{\dis h,\,\dis k,\,\codis (h,k)\bigr\}-\e$.
Then
\begin{gather*}
\dis(h\circ f)\le\dis f+\dis h\le d^c_{GH}(X,Y)+d^c_{GH}(Y,Z)+2\e, \\
\dis(g\circ k)\le\dis g+\dis k\le d^c_{GH}(X,Y)+d^c_{GH}(Y,Z)+2\e, \\
\codis(h\circ f,g\circ k)\le\codis(f,g)+\codis(h,k)\le d^c_{GH}(X,Y)+d^c_{GH}(Y,Z)+2\e,
\end{gather*}
where
$$
d^c_{GH}(X,Z)\le \max\bigl\{\dis (h\circ f),\,\dis (g\circ k),\,\codis (h\circ f,g\circ k)\bigr\}\le d^c_{GH}(X,Y)+d^c_{GH}(Y,Z)+2\e.
$$
It remains to use the arbitrariness of $\e$.
\end{proof}

The proper class of all non-empty metric spaces, considered up to isometry, endowed with the continuous Gromov--Hausdorff distance is denoted by $\GH^c$. It follows from the above that $d^c_{GH}$ is a generalized pseudometric on $\GH^c$. The subset of $\GH^c$ consisting of all compact metric spaces is denoted by $\cM^c$. As in the case of the ordinary Gromov--Hausdorff distance, the restriction of $d^c_{GH}$ to $\cM^c$ is a metric (Theorem~\ref{thm:compGHzero}).

The basic properties of the ordinary Gromov--Hausdorff distance also hold to its continuous analogue.

\begin{prop}\label{prop:GHelementProps}
For any $X,Y\in \GH^c$, the following holds\/\rom:
\begin{enumerate}
\item\label{prop:GHelementProps:1} if metric spaces $X$ and $Y$ are isometric, then $d^c_{GH}(X,Y)=0$\rom;
\item\label{prop:GHelementProps:2} $d_{GH}(X,Y)\le d^c_{GH}(X,Y)$\rom;
\item\label{prop:GHelementProps:3} $2d^c_{GH}(\D_1,X)=\diam X$\rom;
\item\label{prop:GHelementProps:4} $2d^c_{GH}(X,Y)\le\max\{\diam X,\diam Y\}$\rom;
\item\label{prop:GHelementProps:5} if the diameter of $X$ or $Y$ is finite, then $2d^c_{GH}(X,Y)\ge\bigl|\diam X-\diam Y\bigr|$\rom;
\item\label{prop:GHelementProps:6} if the diameter of $X$ is finite, then for any $\l\ge0$, $\mu\ge0$, we have $2d^c_{GH}(\l X,\mu X)=|\l-\mu|\diam X$, whence it follows immediately that the curve $\g(t):=t\,X$ is shortest between any of its points, and the length of such a segment equals the distance between its ends\rom;
\item\label{prop:GHelementProps:7} for any $\l>0$, we have $d^c_{GH}(\l X,\l Y)=\l\,d^c_{GH}(X,Y)$, and if the spaces $X$ and $Y$ are bounded, then the equality also holds for $\l=0$\rom;
\item\label{prop:GHelementProps:8} if $X_1,\,X_2,\ldots$ is a sequence of metric spaces converging in the metric $d^c_{GH}$, then it also converges in the metric $d_{GH}$.
\end{enumerate}
\end{prop}

\begin{proof}
(\ref{prop:GHelementProps:1}) The isometries $h\:X\to Y$ and $h^{-1}$ are continuous and $\dis h=\dis h^{-1}=\codis (h,h^{-1})=0$.

(\ref{prop:GHelementProps:2}) The inequality holds because the distance $d_{GH}(X,Y)$ is the infimum over a larger family of mappings than that needed to calculate $d^c_{GH}(X,Y)$.

(\ref{prop:GHelementProps:3}) For any continuous $\D_1\torllim^f_g X$, we have $\dis f=0$, $\dis g=\diam X$, and if $\D_1=\{p\}$, then
$$
\codis(f,g)=\sup_{x\in X}\Bigl|\bigl|p\,g(x)\bigr|-\bigl|f(p)\,x\bigr|\Bigr|=\sup_{x\in X}\bigl|f(p)\,x\bigr|\le\diam X,
$$
which is what was required.

(\ref{prop:GHelementProps:4}) For any continuous $X\torllim^f_g Y$, we have
\begin{gather*}
\dis f\le\max\{\diam X,\diam Y\}, \ \
\dis g\le\max\{\diam X,\diam Y\}, \\
\codis (f,g)=\sup_{x\in X,\,y\in Y}\Bigl|\bigl|x\,g(y)\bigr|-\bigl|f(x)\,y\bigr|\Bigr|\le \max\{\diam X,\diam Y\},
\end{gather*}
which is what was required.

(\ref{prop:GHelementProps:5}) This follows from the corresponding inequality for $d_{GH}$ and the fact that $d_{GH}\le d^c_{GH}$.

(\ref{prop:GHelementProps:6}) By property~(\ref{prop:GHelementProps:5}) we have
$$
2d^c_{GH}(\l X,\mu X)\ge\bigl|\diam(\l X)-\diam(\mu X)\bigr|=|\l-\mu|\diam X.
$$
To prove the converse inequality, we choose the identity mapping as continuous $X\torllim^f_g X$, then if we denote by $\dis_{\l,\mu}$ and $\codis_{\l,\mu}$ the distortion and codistortion of the mappings between the spaces $\l X$ and $\mu X$, respectively, then
\begin{gather*}
\dis_{\l ,\mu }f=\sup_{x,x'\in X}\Bigl|\l|xx'|-\mu\bigl|xx'\bigr|\Bigr|=|\l-\mu|\diam X,\\
\codis_{\l ,\mu}(f,g)=\sup_{x,x'\in X}\Bigl|\l|xx'|-\mu\bigl|xx'\bigr|\Bigr|=|\l-\mu|\diam X,
\end{gather*}
from which the required result follows.

(\ref{prop:GHelementProps:7}) We choose arbitrary continuous $X\torllim^f_gY$ and denote by $\dis_\l$ and $\codis_\l$ the distortion and co-distortion of these mappings, but between $\l X$ and $\l Y$, respectively. Then
\begin{gather*}
\dis_\l f=\sup_{x,x'\in X}\Bigl|\l|xx'|-\l\bigl|f(x)f(x')\bigr|\Bigr|=\l\,\dis_\l f, \\
\codis_\l (f,g)=\sup_{x\in X,\,y\in Y}\Bigl|\l|x\,g(y)|-\l\bigl|f(x)y\bigr|\Bigr|=\l\,\codis (f,g),
\end{gather*}
from which the declared statement follows.

(\ref{prop:GHelementProps:8}) This follows from the inequality $d_{GH}\le d^c_{GH}$.
\end{proof}

\section{Comparison of Metrics (on Zero-Dimensional Spaces)}

\begin{prop}\label{prop:descret}
If $X$ and $Y$ are discrete metric spaces, then $d^c_{GH}(X,Y)=d_{GH}(X,Y)$.
\end{prop}

\begin{proof}
This follows from the fact that all mappings between $X$ and $Y$ are continuous.
\end{proof}

To prove the following proposition, we need a sufficient condition for the continuity of the mapping of metric spaces.

\begin{lem}\label{lem:SufficientForContinuity}
Let $f\:X\to Y$ be an arbitrary mapping of metric spaces, and there exists a positive number $\e\in\R$ such that for any distinct $y_1,y_2\in Y$ with non-empty preimages and any $x_1\in f^{-1}(y_1)$ and $x_2\in f^{-1}(y_2)$, it holds $|x_1x_2|>\e$. Then the mapping $f$ is continuous.
\end{lem}

\begin{proof}
Indeed, let $y\in Y$ have a non-empty preimage, then for any point $x\in f^{-1}(y)$, we have $U_\e(x)\ss f^{-1}(y)$, therefore $f^{-1}(y)$ is open and, hence, for any open $V\ss Y$, the set $f^{-1}(V)=\cup_{y\in V}f^{-1}(y)$ is also open in $X$.
\end{proof}

\begin{prop}\label{prop:d2}
Let $X$ and $Y$ be arbitrary metric spaces such that $2d_{GH}(X,Y)<s(X)$, then $d_{GH}^c(X,Y)=d_{GH}(X,Y)$.
\end{prop}

\begin{proof}
Let $s:=s(X)$ and let $\e>0$ such that $2d_{GH}(X,Y)<s-\e$. Then there exists a correspondence $R\in\cR(X,Y)$ such that $\dis R<s-\e$. The set of all such correspondences will be denoted by $\cR_\e(X,Y)$.

According to formula~(\ref{eq:3}), the equality $d_{GH}(X,Y)=\frac12\inf\bigl\{\dis R:R\in\cR_\e(X,Y)\bigr\}$ holds. We choose an arbitrary $R\in\cR_\e(X,Y)$. We show that $\bigl|R(x_1)R(x_2)\bigr|>\e$ for $x_1\ne x_2$. Indeed, for any $(x_1,y_1),\,(x_2,y_2)\in R$, $x_1\ne x_2$, the inequality $|y_1y_2|\ge|x_1x_2|-\dis R>s-(s-\e )=\e>0$ holds. Therefore, the sets $\{R(x):x\in X\}$ form an open partition of the space $Y$.

We construct mappings $f\:X\to Y$ and $g\:Y\to X$ as follows: for each $x\in X$, we choose an arbitrarily point $f(x)$ in $R(x)$, and for each $y\in R(x)$, we set $g(y)=x$. Since $s(X)>0$, $X$ is discrete, so $f$ is continuous. The continuity of $g$ follows from Lemma~\ref{lem:SufficientForContinuity}. By Proposition~\ref{prop:CorrespInTermsFuncs}, we have $\dis R_{f,g}\le\dis R$. Thus, we have constructed a mapping $\cR_\e(X,Y)\to\cR_\e(X,Y)$, $R\mapsto R_{f,g}$, therefore, by the definition of $d_{GH}^c$,
$$
d_{GH}^c(X,Y)\le\frac12\inf\bigl\{\dis R_{f,g}:R\in\cR_\e(X,Y)\bigr\}\le
\frac12\inf\bigl\{\dis R:R\in \cR_\e (X,Y)\bigr\}=d_{GH}(X,Y).
$$
It remains to use Inequality~(\ref{prop:GHelementProps:2}) of Proposition~\ref{prop:GHelementProps}.
\end{proof}

\begin{prop}\label{prop:subset}
Let $X$ be a metric space and $\dim X=0$. Then for every non-empty $A\ss X$, we have $d_{GH}^c(A,X)\le d_H(A,X)$.
\end{prop}

\begin{proof}
For an arbitrary $\e>0$, we put $r=d_H(A,X)+\e$, then, by the definition of the Hausdorff distance, $\l=\bigl\{U_r(a):a\in A\bigr\}$ is an open covering of $X$. Since $\dim X=0$, by the theorems of Stone~\cite{Stone} (see also~\cite[Theorem 4.4.1]{E85}) and Dowker~\cite{Dowker} (see also~\cite[Theorem 7.2.4]{E85}), there exists an open covering $\mu=\{U_\b\}$ of multiplicity $1$ inscribed in $\l$, i.e., $\mu$ is a partition of $X$ by open sets. Therefore, for every $U_\b$, there exists a point $a_\b\in A$ such that $U_\b\ss U_r(a_\b)$.

Let $f\:A\to X$ be the inclusion, and let the mapping $g\:X\to A$ be defined by the formula $g(U_\b)=a_\b$. Since the sets $U_\b$ are open and pairwise disjoint, the mapping $g$ is well-defined and continuous, and for every point $x\in X$, the inequality $\bigl|x\,g(x)\bigr|<r$ holds. Clearly, $\dis f=0$,
$$
\dis g=\sup_{x,x'\in X}\Bigl||xx'|-\bigl|g(x)g(x')\bigr|\Bigr|\le\sup_{x,x'\in X}\Bigl(\bigl|x\,g(x)\bigr|+\bigl|x'\,g(x')\bigr|\Bigr)\le 2r,
$$
and finally
$$
\codis (f,g)=\sup_{x\in X,\,a\in A}\Bigl|\bigl|x\,f(a)\bigr|-\bigl|g(x)a\bigr|\Bigr|=
\sup_{x\in X,\,a\in A}\Bigl|\bigl|xa\bigr|-\bigl|g(x)a\bigr|\Bigr|\le\sup_{x\in X}|x\,g(x)|\le r,
$$
therefore $\dis R_{f,g}\le2d_H(A,X)+2\e$ and, because of arbitrariness of $\e$, we have $\dis R_{f,g}\le2d_H(A,X)$, whence $d_{GH}^c(A,X)\le d_{H}(A,X)$, which is what was stated.
\end{proof}

\begin{cor}\label{cor:subset}
For metric spaces $X$, $Y$, and any of their everywhere dense subsets $A\ss\overline{A}=X$, $B\ss\overline{B}=Y$ such that $\dim A=\dim B=0$, we have
$$
d_{GH}^c(A,B)=d_{GH}(X,Y).
$$
\end{cor}

\begin{proof}
For a given $\e>0$, we choose discrete $\e$-nets $A_\e\ss A$ and $B_\e\ss B$ in the sets $A$ and $B$, respectively. According to the triangle inequalities for distances $d_{GH}^c$ and $d_{GH}$, as well as by Propositions~\ref{prop:subset} and~\ref{prop:descret}, the following chain of inequalities holds:
\begin{multline*}
d_{GH}^c(A,B)\le d_{GH}^c(A,A_\e)+d_{GH}^c(A_\e,B_\e)+d_{GH}^c(B_\e,B)\le\\ \le
2\e+d_{GH}^c(A_\e,B_\e)=2\e+d_{GH}(A_\e,B_\e )\le\\ \le
2\e+d_{GH}(A_\e,A)+d_{GH}(A,X)+d_{GH}(X,Y)+d_{GH}(Y,B)+d_{GH}(B,B_\e)\le\\ \le
2\e+\e+0+d_{GH}(X,Y)+0+\e=4\e+d_{GH}(X,Y).
\end{multline*}
From arbitrariness of $\e$ and Inequality~(\ref{prop:GHelementProps:2}) of Proposition~\ref{prop:GHelementProps}, we get
$$
d_{GH}^c(A,B)\le d_{GH}(X,Y)=d_{GH}(A,B)\le d_{GH}^c(A,B),
$$
which completes the proof.
\end{proof}

\begin{cor}\label{cor:dim0}
For any zero-dimensional metric spaces $X$ and $Y$, it holds
$$
d_{GH}^c(X,Y)=d_{GH}(X,Y).
$$
\end{cor}

Let us emphasize that Corollary~\ref{cor:dim0} is a generalization of Proposition~\ref{prop:descret}, which, however, we used in the proof of this more general statement.

\begin{prop}\label{prop:ind}
Let $X$ be a metric space and $\ind X\ne0$. Then there exists $r>0$ such that for every continuous mapping $f\:X\to Y$ into a metric space $Y$ with $\ind Y=0$, the inequality $\dis f\ge r$ holds. In particular, we have $2d_{GH}^c(X,Y)\ge r$.
\end{prop}

\begin{proof}
By the condition $\ind X\ne0$, there exist a closed subset $F\ss X$ and a point $x\in X\sm F$ such that every clopen set containing the point $x$, intersects the set $F$. Put $r=|xF|>0$. For every mapping $f\:X\to Y$, the alternative holds: either $\bigl|f(x)f(F)\bigr|=0$, or $\bigl|f(x)f(F)\bigr|>0$.

In the first case, the estimate $\dis f\ge|xF|=r$ is valid.

In the second case, the condition $\ind Y=0$ implies the existence of a clopen neighborhood $U^{f(x)}$ of the point $f(x)$ that lies outside the closure $\overline{f(F)}$ of the set $f(F)$. The preimage $U^x:=f^{-1}\bigl(U^{f(x)}\bigr)$ is a clopen neighborhood of the point $x$, and $U^x\cap F=\0 $, which is a contradiction with the choice of the point $x$ and the closed set $F$.
\end{proof}

\begin{prop}\label{prop:Ind}
Let $X$ be a metric space and $\Ind X\ne0$. Then there exists $r>0$ such that for every continuous mapping $f\:X\to Y$ into a metric space $Y$ with $\Ind Y=0$, the inequality $\dis f\ge r$ holds. In particular, we have $2d_{GH}^c(X,Y)\ge r$.
\end{prop}

\begin{proof}
By the Nagami--Roberts Theorem~\cite[Theorem on p.~601]{NR65} there exist closed subsets $F,H\ss X$ such that $r=|FH|>0$ and every clopen set containing $F$ intersects $H$. For every mapping $f\:X\to Y$, the alternative holds: either $\bigl|f(F)f(H)\bigr|=0$, or $\bigl|f(F)f(H)\bigr|>0$.

In the first case, the estimate $\dis f\ge|FH|=r$ takes place.

In the second case, the condition $\Ind Y=0$ implies the existence of a clopen neighborhood $U^{f(F)}$ of the closed set $\overline{f(F)}$ that lies outside the closure $\overline{f(H)}$ of the set $f(H)$. The preimage $U^F=f^{-1}\bigl(U^{f(F)}\bigr)$ is a clopen neighborhood of the set $F$ and $U^F\cap H=\0$, which is a contradiction with the choice of the closed sets $F$ and $H$.
\end{proof}

Recall that a topological space is called \emph{totally disconnected} if all its connected components are singletons.

\begin{prop}\label{prop:ConnectDiscrete}
Let $K\ss X$ be a connected subset of a metric space $X$, and $Y$ be a totally disconnected metric space. Then for every continuous mapping $f\:X\to Y$, we have $\dis f\ge \diam K$ and therefore $2d^c_{GH}(X,Y)\ge \diam K$. In particular, if $X$ is a connected metric space, and $Y$ is a totally disconnected metric space such that $\diam X\ge \diam Y$, then $2d^c_{GH}(X,Y)=\diam X$. For example, this holds for any discrete $\e$-net $Y\ss X$.
\end{prop}

\begin{proof}
If $f\:X\to Y$ is a continuous mapping, then, by virtue of the connectedness of $K$, the restriction of $f$ to $K$ is constant,
so that $\dis f\ge \diam K$, therefore $2d^c_{GH}(X,Y)\ge \diam X$ by the definition of continuous Hausdorff distance.
The second assertion follows from Item~(\ref{prop:GHelementProps:2}) of Proposition~\ref{prop:GHelementProps}.
\end{proof}

\section{Zero Distance}

We have already pointed out that the (continuous) Gromov--Hausdorff distance is a pseudometric, i.e., it can be zero between non-isometric spaces. Therefore, it is important to describe certain classes of metric spaces that are at zero distance from each other.

\begin{thm}\label{thm:compGHzero}
For compact metric spaces $X$ and $Y$, the following conditions are equivalent\/\rom:
\begin{enumerate}
\item\label{thm:compGHzero:1} $d_{GH}^c(X,Y)=0$\rom;
\item\label{thm:compGHzero:2} $d_{GH}(X,Y)=0$\rom;
\item\label{thm:compGHzero:3} spaces $X$ and $Y$ are isometric.
\end{enumerate}
\end{thm}

\begin{proof}
The implications (\ref{thm:compGHzero:3})\imply(\ref{thm:compGHzero:1})\imply(\ref{thm:compGHzero:2}) follow from Properties~(\ref{prop:GHelementProps:1}) and~(\ref{prop:GHelementProps:2}) of Proposition~\ref{prop:GHelementProps}, respectively.

Implication (\ref{thm:compGHzero:2})\imply(\ref{thm:compGHzero:3}) is the famous deep Theorem~\cite[theorem 7.3.30]{BurBurIva01}.
\end{proof}

For applications, we will need the following version of the previous theorem, which follows immediately from~\cite[Exercise 7.3.31]{BurBurIva01}.

\begin{thm}\label{thm:equivcompGHzero}
If a metric space $X$ is compact, then the following conditions on a metric space $Y$ are equivalent\rom:
\begin{enumerate}
\item $d_{GH}(X,Y)=0$\rom;
\item the space $Y$ is isometric to a dense subset of the space $X$ \(e.g., in the case of $X=S^n$ there are at least continuum number of pairwise non-isometric $Y$, see Introduction\/\).
\end{enumerate}
\end{thm}

The situation with the continuous Gromov--Hausdorff distance is significantly different.

We fix on the sphere $S^n$, $n\ge 1$, some metric $\r$ and some continuous free involution $\s$ (for example, an antipodal mapping). Since each continuous function on a compact set attains its minimum, then
$$
d(\r,\s):=\inf\Bigl\{\r\bigl(x,\s(x)\bigr):x\in S^n\Bigr\}=\min\Bigl\{\r\bigl(x,\s(x)\bigr):x\in S^n\Bigr\}>0.
$$
We put
$$
d(\r):=\sup\bigl\{d(\r,\s):\text{$\s$ --- continuous free involution on $S^n$}\bigr\}>0.
$$

\begin{prop}\label{prop:sphere}
For an arbitrary metric $\r$ on the sphere $S^n$ and an arbitrary metric space $Y$ that is topologically embeddable in $\R^n$, the following inequalities hold\/\rom:
$$
d(\r)\le2d^c_{GH}(S^n,Y)\le\max\{\diam S^n,\diam Y\}.
$$
\end{prop}

\begin{proof}
Let $\nu\:Y\to\R^n$ be a topological embedding, $f\:S^n\to Y$ and $g\:Y\to S^n$ be continuous mappings, and $h=\nu\circ f$. Consider an arbitrary continuous free involution $\s$ on the sphere $S^n$. According to generalized theorem of Borsuk--Ulam type~\cite[Theorem 5]{Fet}, there exists a point $x\in S^n$ such that $h(x)=h\bigl(\s(x)\bigr)$, therefore, since $\nu$ is injective, we have $f(x)=f\bigl(\s(x)\bigr)$. This means that $\dis f\ge d(\r,\s)$. Consequently, $d(\r,\s)\le 2d^c_{GH}(S^n,Y)$. The left inequality follows from the arbitrariness of the involution under consideration.

The right inequality is Property~(\ref{prop:GHelementProps:4}) of Proposition~\ref{prop:GHelementProps}.
\end{proof}

\begin{rk}\label{rk:sphere}
From Proposition~\ref{prop:sphere} it follows that for the standard sphere of unit radius $S^n\ss \R^n$, $n\ge 1$, endowed with a geodesic or induced Euclidean distance, for every $m>n$, the equality $2d^c_{GH}(S^n,S^m)=\diam S^n$ holds. Note that according to~\cite[Theorem A, p. 5]{LimMemoliSmith}, we get $2d_{GH}(S^n,S^m)<\diam S^n$.
\end{rk}

\begin{cor}\label{cor:sphere}
For an arbitrary metric $\r$ on the sphere $S^n$ and an arbitrary metric space $Y$, the following conditions are equivalent\/\rom:
\begin{enumerate}
\item\label{cor:sphere:1} $d_{GH}^c(S^n,Y)=0$, and
\item\label{cor:sphere:2} the space $Y$ is isometric sphere $(S^n,\r)$.
\end{enumerate}
\end{cor}

\begin{proof}
(\ref{cor:sphere:1})\imply(\ref{cor:sphere:2}).
Since $d_{GH}(S^n,Y)\le d_{GH}^c(S^n,Y)=0$, then according to Theorem~\ref{thm:equivcompGHzero}, one can assume that the space $Y$ is contained in $(S^n,\r)$. If $Y$ is a proper subset of the sphere, then it is topologically embedded in the space $\R^n$ and $2d^c_{GH}(S^n,Y)\ge d(\r)>0$ according to Proposition~\ref{prop:sphere}.

The implication (\ref{cor:sphere:2})\imply(\ref{cor:sphere:1}) is obvious.
\end{proof}

We say that a metric space $X$ \emph{has the uniqueness property in the class $\GH$, $\cM$, $\GH^c$, or $\cM^c$} if for every metric space $Y$ in the class under consideration, the equality $d_{GH}(X,Y)=0$ in the first two classes and the equality $d_{GH}^c(X,Y)=0$ in the second two classes implies that $X$ and $Y$ are isometric. Clearly, Theorem~\ref{thm:compGHzero} and Corollary~\ref{cor:sphere} describe some classes of spaces with the uniqueness property in the classes $\cM^c$, $\cM$, and $\GH^c$, respectively.

\begin{prop}\label{prop:completenondiscrete}
If a metric space $X$ has the property of uniqueness in the class $\GH$, then it is complete and discrete.
\end{prop}

\begin{proof}
Let $X$ be a non-complete space. Consider its completion $Y$. Since $X$ can be isometrically identified with a dense subset of $Y$, we have $d_{GH}(X,Y)\le d_{H}(X,Y)=0$. The spaces $X$ and $Y$ themselves are not isometric, since the former is non-complete and the latter is complete.

Let the complete space $X$ be non-discrete. Let us take a non-isolated point $x_0\in X$ in it. Then $d_{GH}(X,X\sm x_0)=0$, but the spaces $X$ and $X\sm x_0$ are not isometric, because the first is complete and the second is not.
\end{proof}

\begin{examp}\label{exBBRT}
In~\cite[Example 5.11]{BBRT25} an example of a countable complete bounded metric space $Y$ and its clopen subset $X$ with the following properties is constructed.
\begin{enumerate}[leftmargin=*,itemsep=0.3cm]
\item There is exactly one non-isolated point in the space $Y$.
\item The space $X$ is discrete. Therefore, the space $Y$ cannot be topologically (and therefore isometrically) embedded in the space $X$.
\item $d_{GH}(X,Y)=0$, and therefore $d^c_{GH}(X,Y)=0$ by Corollary~\ref{cor:dim0}.
However, for these spaces, the corresponding mappings $f_\e\:X\to Y$ and $g_\e\:Y\to X$ can easily be presented constructively.
\item
\begin{itemize}
\item $0=s(Y)=s(X)<d(X)=d(Y)=3.5$,
\item $3=R(Y)=R(X)<\diam X=\diam Y=4$,
\item $d_H(X,Y\sm X)=3$,
\item $\bigl|X(Y\sm X)\bigr|=2$.
\end{itemize}
\item For each isometric embedding $h\:X\to Y$, it holds 
$$
d_H\bigl(h(X),Y\sm h(X)\bigr)=3\ \ \text{and}\ \  2\le\bigl|h(X)\bigl(Y\sm h(X)\bigr)\bigr|\le3.
$$
\end{enumerate}
\end{examp}

Spaces $X$ and $Y$ show that Proposition~\ref{prop:completenondiscrete} cannot be conversed.

It is clear that $s(X)>0$ implies that the space $X$ is complete and discrete. The converse, generally speaking, does not hold.

\begin{examp}\label{excompletedescrete}
On the line $\R$, we consider the subset $X=\bigl\{n\pm\frac1{6n}:n\in\N\bigr\}$. The space $X$ is complete, discrete, and $s(X)=0$.
\end{examp}

For a metric space $X$, consider the set of all distances in it: $\dist X=\bigl\{|xx'|:x,x'\in X\bigr\}$.

It is easy to check that
\begin{itemize}
\item $s(X)=\Bigl|0\,\bigl(\dist X\sm\{0\}\bigr)\Bigr|$ and $\diam X=d_H\bigl(\{0\},\dist X\bigr)=\diam\dist X$;
\item if the spaces $X$ and $Y$ are isometric, then $\dist X=\dist Y$;
\item if $d_{GH}(X,Y)=0$, then $d_H\bigl(\dist X,\dist Y\bigr)=0$, i.e., $\overline{\dist X}=\overline{\dist Y}$ (see also~\cite[Lemma 5.1]{VikhrovSerb}).
\end{itemize}

\begin{prop}\label{propcomparison}
For any $X,Y\in \GH$, the following holds\/\rom:
\begin{enumerate}
\item\label{propcomparison:1} $d_H\bigl(\dist X,\dist Y\bigr)\le2d_{GH}(X,Y)$\rom;
\item\label{propcomparison:2} $s(X)\le2d_{GH}(X,Y)$, or $s(Y)\le s(X)+2d_{GH}(X,Y)$\rom;
\item\label{propcomparison:3} $2s(X)<s(Y)$, or $s(Y)-s(X)\le d_H(\dist X,\dist Y)$.
\end{enumerate}
\end{prop}

\begin{proof}
If $d_{GH}(X,Y)=\infty$, then Items~(\ref{propcomparison:1}) and~(\ref{propcomparison:2}) are obvious (in Item~(\ref{propcomparison:3}) the distance $d_{GH}(X,Y)$ is not contained). Therefore, we will assume that $d_{GH}(X,Y)<\infty$.

(\ref{propcomparison:1}) Fix a number $\e >0$ and a correspondence $R\in\cR(X,Y)$ such that $\dis R<2d_{GH}(X,Y)+\e$. Take two arbitrary points $x,x'\in X$, and let $y,y'\in Y$ be points such that $(x,y),\ (x',y')\in R$. Then $\bigl||xx'|-|yy'|\bigr|\le\dis R<2d_{GH}(X,Y)+\e$. The first inequality follows from the obtained inclusion $\dist X\ss B_{2d_{GH}(X,Y)+\e}(\dist Y)$, the arbitrariness of $\e >0$, and the symmetry of the condition under consideration.

(\ref{propcomparison:2}) Let $s(X)>2d_{GH}(X,Y)$. Since the second inequality is automatically satisfied for $s(Y)\le s(X)$, we will assume that $s(Y)>s(X)$. Then
\begin{multline*}
\bigl|s(X)\,0\bigr|>2d_{GH}(X,Y)\ge d_H(\dist X,\dist Y)\ge\bigl|s(X)\dist Y\bigr|=\\ =
\min\Bigl\{\bigl|s(X)\,0\bigr|,s(Y)-s(X)\Bigr\}=s(Y)-s(X),
\end{multline*}
which is what was required.

(\ref{propcomparison:3}) If $s(Y)\le s(X)$, then the second inequality is obvious. Now consider the remaining case $s(X)<s(Y)\le2s(X)$. Then $\bigl|s(X)s(Y)\bigr|=s(Y)-s(X)\le s(X)=\bigl|s(X)\,0\bigr|$, and taking into account that $\bigl|s(X)\dist Y\bigr|\le d_H(\dist X,\dist Y)$, we obtain
$$
\bigl|s(X)\dist Y\bigr|=\min\Bigl\{\bigl|s(X)\,0\bigr|,\,\bigl|s(X)s(Y)\bigr|\Bigr\}=\bigl|s(X)s(Y)\bigr|\le d_H(\dist X,\dist Y),
$$
which is what was claimed.
\end{proof}

If the diameter of $X$ or $Y$ is finite, then $2d_{H}(X,Y)\ge \bigl|\diam X-\diam Y\bigr|$, so Inequality~(\ref{propcomparison:1}) of Proposition~\ref{propcomparison} is a strengthening of Property~(\ref{prop:GHelementProps:5}) of Proposition~\ref{prop:GHelementProps}. Note also that Property~(\ref{propcomparison:1}) implies

\begin{cor}\label{cors}
For any $X,Y\in\GH$, from $d_{GH}(X,Y)=0$ it follows that $s(X)=s(Y)$.
\end{cor}

\begin{thm}\label{thm:GHzeros}
If $s(X)>0$ and $d_{GH}(X,Y)=0$, then for any $\e>0$, there exists a homeomorphism $f_\e\:X\to Y$ such that $\bigl||f_\e(x)f_\e(x')|-|xx'|\bigr|\le\e$ for any points $x,x'\in X$.
\end{thm}

\begin{proof}
Let $d_{GH}(X,Y)=0$. By Proposition~\ref{prop:d2}, we have $d_{GH}^c(X,Y)=0$. Let $0<\e <s(X)$ and $f_\e\:X\to Y$, $g_\e\:Y\to X$ satisfy $\dis R_{f_\e ,g_\e }\le\e$.

We show that $g_\e\circ f_\e\:X\to X$ and $g_\e\circ f_\e\:Y\to Y$ are identity mappings. For an arbitrary point $x\in X$ and the point $y=f_\e(x)\in Y$, we write the codistortion
$$
\Bigl|x\,g_\e\bigl(f_\e (x)\bigr)\Bigr|=\Bigl|\bigl|x\,g_\e(y)\bigr|-\bigl|f_\e(x)\,y\bigr|\Bigr|\le\e<s(X).
$$
Therefore, $x=g_\e\bigl(f_\e(x)\bigr)$.

Corollary~\ref{cors} implies $s(Y)=s(X)$. Therefore, the equality $g_\e\circ f_\e=\Id_Y$ can be proved similarly.

The inequality $\bigl||f_\e (x)f_\e (x')|-|xx'|\bigr|\le\e$ is exactly the inequality $\dis f_\e\le\e$.
\end{proof}

\begin{prop}\label{propbigdifferent}
Let $X\in\GH$ have a number $r>0$ such that $s(\dist X)\ge r$ \(the distance between distinct points of the set $\dist X$ is at least $r$\). Then $s(X)\ge r$, and for $\e<r$, every mapping $f_\e$ from Theorem~$\ref{thm:GHzeros}$ is an isometry. In particular, this space $X$ has the uniqueness property in the class $\GH$.
\end{prop}

\begin{proof}
It is clear that the set $\dist X\ss\R$ is closed and discrete. Let $d_{GH}(X,Y)=0$. Then from Item~(\ref{propcomparison:1}) of Proposition~\ref{propcomparison} it follows $d_H\bigl(\dist X,\dist Y\bigr)=0$, whence, since $\dist X$ is closed, we obtain $\dist X=\overline{\dist X}=\overline{\dist Y}$, therefore $\dist Y\ss\dist X$. Since $\dist X$ is discrete, the closure of a proper subset cannot be equal to $\dist X$, whence $\dist X=\dist Y$. According to the choice of the mapping $f_\e\:X\to Y$, for any points $x,x'\in X$, the inequality $\bigl||f_\e(x)f_\e(x')|-|xx'|\bigr|\le\e<r$ holds, and if $\bigl|f_\e(x)f_\e(x')\bigr|\ne|xx'|$, then the quantities $\bigl|f_\e(x)f_\e(x')\bigr|$ and $|xx'|$ differ from each other by at least $r$. The latter, together with the previous inequality, implies $\bigl|f_\e(x)f_\e(x')\bigr|=|xx'|$.
\end{proof}

\begin{cor}\label{corfinite}
Every metric space $X$ with a finite set $\dist X$ has the property of uniqueness in the class $\GH$.
\end{cor}

\begin{rk}
The Corollary~\ref{corfinite} can also be immediately obtained from~\cite[Lemmas 5.1 and 6.1]{VikhrovSerb}.
\end{rk}

\begin{conj}\label{conjbigdifferent}
Every metric space $X$ with a closed and discrete set $\dist X$ has the property of uniqueness in the class $\GH$.
\end{conj}

\begin{examp}\label{exs(X)>0}
On $X=\{0\}\sqcup\N\sqcup\{\infty\}$ we define the following metric:
$$
|xx'|=
\begin{cases}
2& \text{at $0\ne x\ne x'\ne 0$};\\
1+\frac1{2n}& \text{for $x=0$, $x'\in\N\sqcup\{\infty\}$}.
\end{cases}
$$
The metric space $X$ and its clopen subset $Y=\{0\}\sqcup\N\ss X$ have the following properties.
\begin{enumerate}[leftmargin=*,itemsep=0.3cm]
\item $\diam X=2=\diam Y$ and $s(X)=s(Y)=1$, therefore the spaces $X$ and $Y$ are complete and discrete.
\item $\dist X=\{0,1,1+\frac1{2n},2:n\in \N\}\ne\{0,1+\frac1{2n},2:n\in\N\}=\dist Y$, therefore the spaces $X$ and $Y$ are not isometric.
\item $d_{GH}^c(X,Y)=0$. The mapping $h_m\:X\to Y$, defined by the formula
$$
h_m(x)=
\begin{cases}
0& \text{when $x=0$};\\
x& \text{for $1\le x<m$};\\
x+1& \text{for $m\le x<\infty$};\\
m& \text{when $x=\infty$},
\end{cases}
$$
is a $\frac1{2m}$- isometry .
\item The isometries of the spaces $X$ and $Y$ are only their identity mappings.
\item $\bigl|Y\,(X\sm Y)\bigr|=1$ and $d_H(X,Y)=1$.
\end{enumerate}
\end{examp}

\section{Comparison of Topologies}
Let $p_c\:\GH^c\to \GH$ be the identity mapping. By Property~(\ref{prop:GHelementProps:1}) of Proposition~\ref{prop:GHelementProps}, $p_c$ is non-expanding and hence continuous. Recall also that a continuous bijective mapping is called a \emph{condensation}. This concept also carries over to topological classes, so $p_c$ is a condensation. Recall (Proposition~\ref{prop:ConnectDiscrete}) that in this case, two metric spaces at positive distance $d_{GH}^c$ can be sometimes mapped to spaces with zero distance $d_{GH}$.

\begin{prop}\label{prop:DiscontAtIndPosit}
The mapping $p_c^{-1}$ is discontinuous at every point $X\in \GH$ such that $\Ind X\ne0$.
\end{prop}

\begin{proof}
Let us choose an arbitrary non-empty metric space $X$ such that $\Ind X\ne0$. By Proposition~\ref{prop:Ind}, there exists $r>0$ such that for every non-empty metric space $Y$ with $\Ind Y=0$, we have $d^c_{GH}(X,Y)\ge r$. For every discrete space $Z$, it holds $\Ind Z=0$. On the other hand, it is easy to show that for every $\dl>0$, there exists a discrete $Y\ss X$ such that $d_{GH}(X,Y)\le d_H(X,Y)<\dl$. The latter means that there is no a neighborhood $U_\dl(X)$ with $\dl<r$ and $p_c^{-1}\bigl(U_\dl(X)\bigr)\ss U_r(X)$. The latter means the discontinuity of $p_c^{-1}$ at the point $X$.
\end{proof}

\begin{rk}\label{rk:finite}
In the proof of Proposition~\ref{prop:DiscontAtIndPosit} we actually showed more, namely that $p_c^{-1}$ is discontinuous at $X$, $\Ind X\ne0$ on various smaller sets,
for example,
\begin{itemize}
\item on the union of $\{X\}$ with the class of all spaces with zero large inductive dimension;
\item on the union of $\{X\}$ with the class of all discrete spaces;
\item for compact $X$, on the union of $\{X\}$ with the set of all finite metric spaces.
\end{itemize}

Recall that the mapping $p_c$ is an isometry on the subclass of all zero-dimensional spaces in the sense of Lebesgue (Corollary~\ref{cor:dim0}). However, the results obtained above show that adding even a single point $X$ can lead to discontinuity in $p_c^{-1}$ and, hence, to violation of isometry.
\end{rk}

Below we will show that the mapping $p_c^{-1}$ can be discontinuous at points corresponding to discrete spaces, and the discontinuity manifests itself on a sufficiently thin subset of $\GH$. Nevertheless, on completely discrete spaces $X$, i.e., when $s(X)>0$, the mapping $p_c^{-1}$ is continuous. The following result follows immediately from Proposition~\ref{prop:d2}.

\begin{cor}\label{cor:isometry}
If $s(X)>0$, then the mapping $p_c^{-1}$ is continuous at the point $X\in\GH$.
\end{cor}

For compact metric spaces, there is a natural criterion for the continuity of the mapping $p_c^{-1}$.

\begin{thm}\label{thm:comp0}
For a compact metric space $X\in\GH$, the following conditions are equivalent
\begin{enumerate}
\item\label{thm:comp0:1} $\Ind X=0$;
\item\label{thm:comp0:2} the mapping $p_c^{-1}$ is continuous at $X$.
\end{enumerate}
\end{thm}

\begin{proof}
(\ref{thm:comp0:2})\imply(\ref{thm:comp0:1}). Assume the contrary, i.e., that $\Ind X\ne0$. By Proposition~\ref{prop:Ind}, there exists $r>0$, depending only on $X$, such that $2d_{GH}^c(X,Y)\ge r$. But the latter contradicts the continuity of $p_c^{-1}$ at $X$.

(\ref{thm:comp0:1})\imply(\ref{thm:comp0:2}). We need to prove that for every $\e>0$, there exists $\dl >0$ such that $d_{GH}(X,Y)<\dl$ implies $d_{GH}^c(X,Y)<\e$. For $0<\e'<\e/4$, consider the covering $\l=\bigl\{U_{\e'}(x)\bigr\}_{x\in X}$. Since $\dim X=0$, there exists a clopen partition $\mu $ inscribed in $\l$. Since $X$ is compact, the partition $\mu$ is finite. Let $\mu=\{U_1,\ldots,U_m\}$, then $r:=\min\bigl\{|U_iU_j|:i\ne j\bigr\}>0$ and $\diam U_i\le2\e'<\e/2$. We set $\dl =\min\{r/2,\e'\}<\e/4$ and show that $\dl$ is what we are looking for.

Let $d_{GH}(X,Y)<\dl$, then there exists $\a>0$ such that $d_{GH}(X,Y)<\dl-\a$ and, therefore, we can choose a correspondence $R\in\cR(X,Y)$ for which $\dis R<2\dl-2\a\le r-2\a$. Let $V_k=R(U_k)$, then for $i\ne j$, we have $V_i\cap V_j=\0$, since for any $y_k\in V_k$ and $x_k\in U_k$ such that $(x_k,y_k)\in R$, we have
$$
|y_iy_j|\ge|x_ix_j|-\dis R>r-(r-2\a )=2\a.
$$
Thus, $V_k=\cup_{y\in V_k}U_\a(y)$ is an open set for every $k$, so $\{V_k\}_{k=1}^m$ is a partition of $Y$ by clopen sets. Note that $\diam V_k\le\diam U_k+\dis R\le 2\e'+\e'<3\e/4$.

Let us now choose arbitrary $x_k\in U_k$, $y_k\in V_k$, $(x_k,y_k)\in R$ and define the mappings $f\:X\to Y$ and $g\:Y\to X$ as follows: $f(U_k)=y_k$ and $g(V_k)=x_k$, then $f$ and $g$ are continuous. Let us show that $\dis R_{f,g}<2\e$, whence $d_{GH}^c(X,Y)\le \frac12\dis R_{f,g}<\e$, which is what we want.

Let us choose arbitrary $x\in U_i$ and $x'\in U_j$, then
\begin{multline*}
\bigl||xx'|-|y_iy_j|\bigr|=\bigl||xx'|-|x'x_i|+|x'x_i|-|x_ix_j|+|x_ix_j|-|y_iy_j|\bigr|\le\\ \le
\bigl||xx'|-|x'x_i|\bigr|+\bigl||x'x_i|-|x_ix_j|\bigr|+\bigl||x_ix_j|-|y_iy_j|\bigr|\le\\ \le
|xx_i|+|x'x_j|+\dis R\le \diam U_i+\diam U_j+\dis R<\e/2+\e/2+\e/4<2\e,
\end{multline*}
therefore $\dis f<2\e$.

Let us now choose arbitrary $y\in V_i$ and $y'\in V_j$, then
\begin{multline*}
\bigl||x_ix_j|-|yy'|\bigr|=\bigl||x_ix_j|-|y_iy_j|+|y_iy_j|-|y_jy|+|y_jy|-|yy'|\bigr|\le\\ \le
\bigl||x_ix_j|-|y_iy_j|\bigr|+\bigl||y_iy_j|-|y_jy|\bigr|+\bigl||y_jy|-|yy'|\bigr|\le\\ \le
\dis R+|y_iy|+|y_jy'|\le\dis R+\diam V_i+\diam V_j<\e/4+3\e/4+3\e/4<2\e,
\end{multline*}
therefore $\dis g<2\e$.

Finally, let us estimate the codistortion $\codis (f,g)$. To do this, we choose arbitrary $x\in U_i$ and $y\in V_j$, then
\begin{multline*}
\bigl||x\,x_j|-|y_iy|\bigr|=\bigl||x\,x_j|-|x_jx_i|+|x_jx_i|-|y_jy_i|+|y_jy_i|-|y_iy|\bigr|\le\\ \le
\bigl||x\,x_j|-|x_jx_i|\bigr|+\bigl||x_jx_i|-|y_jy_i|\bigr|+\bigl||y_jy_i|-|y_iy|\bigr|\le\\ \le
|xx_i|+\dis R+|y_jy|\le \diam U_i+\dis R+\diam U_j<\e/2+\e/4+\e/2<2\e,
\end{multline*}
from which $\codis(f,g)<2\e$ and, therefore, $\dis R_{f,g}<2\e$, which completes the proof.
\end{proof}

\begin{examp}\label{examp:ex1}
Theorem~\ref{thm:comp0} states that at a point $X$, which is a zero-dimensional compact (complete and totally bounded) space, the mapping $p_c^{-1}$ is continuous. It turns out that if we abandon the condition of total boundedness, then even for countable discrete spaces, continuity is not guaranteed. Below we give an example of a countable (and therefore zero-dimensional) complete discrete space $X$ such that the mapping $p_c^{-1}$ is discontinuous at the point $X\in \GH$. By Corollary~\ref{cor:isometry}, $s(X)=0$ for this $X$.

Let us take a countable number $Z=I\times\N=\sqcup_{n\in\N}I_n$ of standard unit intervals $I_n=[0,1]$. We define the distance between points of one interval to be standard, and set the distance between points of different intervals to be equal to $1$. In the interval $I_n$, we consider the subset $S_n=\bigl\{i/2^n:i=0,\ldots,2^n\bigr\}$. We put
$$
X_\infty=\sqcup_{k\in\N}S_k\ss Z\ \ \text{i}\ \ X_n=\bigl(\sqcup_{k<n}S_k\bigr)\sqcup\bigl(\sqcup_{k\ge n}I_k\bigr)\ss Z.
$$
Note that $X_\infty$ is a countable discrete space, so it is complete. It is easy to see that $d_H(X_\infty,X_n)=2^{-(n+1)}$, so $X_n\tolim^{d_{GH}\strut}X_\infty$. Furthermore, since $X_\infty$ is a totally disconnected space, and the maximum diameter of connected components in $X_n$ is $1$, by Proposition~\ref{prop:ConnectDiscrete} we get $2d_{GH}^c(X_\infty,X_n)\ge1$, so $X_n$ does not converge to $X_\infty$ w.r.t\. $d_{GH}^c$, and hence the mapping $p_c^{-1}$ is discontinuous at $X_\infty$.
\end{examp}

\section{Incomparable Spaces}

We say that a topological space $X$ is \emph{incomparable\/} with $Y$ if every continuous mapping $f\:X\to Y$ is trivial, i.e., there is a point $y_f\in Y$ such that $f(X)=y_f$. For example, every connected space is incomparable with every totally disconnected space. Note that the incomparability relation is not symmetric; for example, the interval $X=[0,1]$ is incomparable with the set $Y$ of its rational points, but $Y$ is comparable with $X$ (the inclusion of $Y$ in $X$ is a nontrivial continuous mapping).

\begin{prop}\label{prop:2}
Let the metric space $X$ be incomparable with the metric space $Y$. Then
\begin{enumerate}
\item\label{prop:2:1} for any mapping $f\:X\to Y$, it holds $\dis f=\diam X$\rom;
\item\label{prop:2:2} for any mappings $f\:X\to Y$ and $g\:Y\to X$, it holds
$$
\codis(f,g)\ge\max\bigl\{R(X),R(Y)\bigr\}.
$$
\end{enumerate}
\end{prop}

\begin{proof}
(\ref{prop:2:1}) We have
$$
\dis f=\sup_{x,x'\in X}\bigl||xx'|-|f(x)f(x')|\bigr|=\sup_{x,x'\in X}|xx'|=\diam (X).
$$

(\ref{prop:2:2}) We have
\begin{multline*}
\codis (f,g)=\sup_{x\in X,\,y\in Y}\bigl||x\,g(y)|-|f(x)y|\bigr|\ge
\sup_{x\in X}\bigl||x\,g(y_f)|-|f(x)y_f|\bigr|=\\ =
\sup_{x\in X}|x\,g(y_f)|=R_{g(y_f)}(X)\ge R(X).
\end{multline*}
Similarly,
\begin{multline*}
\codis (f,g)=\sup_{x\in X,y\in Y}\Bigl|\bigl|x\,g(y)\bigr|-\bigl|f(x)y\bigr|\Bigr|\ge
\sup_{y\in Y}\Bigl|\bigl|g(y)g(y)\bigr|-\bigl|f\bigl(g(y)\bigr)y\bigr|\Bigr|=\\
\sup_{y\in Y}\Bigl|\bigl|g(y)g(y)\bigr|-|y_fy|\Bigr|=
\sup_{y\in Y}|y_fy|=R_{y_f}(Y)\ge R(Y).
\end{multline*}
To complete the proof, it remains to gather together the two resulting inequalities.
\end{proof}

\begin{cor}\label{cor:2}
Let a metric space $X$ be incomparable with a metric space $Y$. Then $2d_{GH}^c(X,Y)\ge\diam(X)$.
\end{cor}

\begin{cor}\label{cor:2=}
Let $X$ and $Y$ be mutually incomparable metric spaces. Then
$$
2d_{GH}^c(X,Y)=\max\bigl\{\diam(X),\diam(Y)\bigr\}.
$$
\end{cor}

\begin{proof}
The inequality $2d_{GH}^c(X,Y)\ge \max\bigl\{\diam(X),\diam(Y)\bigr\}$ follows from Corollary~\ref{cor:2}. The inverse inequality is Item~(\ref{prop:GHelementProps:4}) of Proposition~\ref{prop:GHelementProps}.
\end{proof}

\section{Hyperspace of Cook's Continuum}

Recall that a \emph{continuum} is every connected compact Hausdorff topological space. We will restrict ourselves to continua that are metric spaces. Thus, in what follows, by \emph{continuum} we mean a metrizable continuum.

We will also be interested in various \emph{hyperspaces}. If $X$ is a metric space, then by $\cH(X)$ we denote the family of all non-empty closed bounded subsets of $X$, endowed with the Hausdorff distance $d_H$. It is well known~\cite{BurBurIva01} that $d_H$ is a metric on $\cH(X)$. This $\cH(X)$ will be in our case the most general example of hyperspaces (a richer list of hyperspaces can be found in~\cite{IllanesNadler}). An important subspace of $\cH(X)$ is the family $\cK(X)$ of all non-empty compact subsets of $X$. If $X$ is compact, then $\cK(X)=\cH(X)$. Another ``restriction'' of hyperspace is obtained if we restrict ourselves to subcontinua of $X$. The family of all subcontinua in $X$ is denoted by $\cCK(X)$. Since every point in $X$ is a degenerate subcontinuum and the Hausdorff distance $d_H$ between one-point subspaces of $X$ coincides with the distance in $X$ between the corresponding points, there is an isometric embedding $x\mapsto\{x\}$ from $X$ into $\cCK(X)\ss\cK(X)\ss \cH(X)$. In what follows, we will informally write $X\ss\cCK(X)$ to refer to this embedding.

\begin{thm}\label{thm:XandCKclosed}
The spaces $X$ and $\cCK(X)$ are closed subsets of $\cK(X)$.
\end{thm}

\begin{thm}[\cite{BurBurIva01}]\label{thm:CompactInherited}
A space $\cH(X)$ is complete\/ \(totally bounded, compact, boundedly compact\) if and only if $X$ is of the same type.
\end{thm}

If $X$ is a continuum, then, by Theorem~\ref{thm:CompactInherited}, the space $\cK(X)=\cH(X)$ is also compact. By Theorem~\ref{thm:XandCKclosed}, $\cCK(X)$ is a closed subset of the compact set $\cK(X)$, and therefore is itself compact. A less trivial fact is that the connectedness of $X$ is also inherited. And even more interesting: $\cCK(X)$ turns out to be path-connected, and for a non-degenerate $X$, one can estimate the dimension of $\cCK(X)$.

\begin{thm}\label{thm:cCKlinearConnectedCont}
For any continuum $X$, the space $\cCK(X)$ is linearly connected continuum\/ {\rm(\cite[Theorem 14.9]{IllanesNadler})}, has trivial shape\/ {\rm(\cite[Theorem 19.10]{IllanesNadler})} and, in particulars, is acyclical in each dimension\/ {\rm(\cite[Theorem 19.3]{IllanesNadler})}.

For a hereditarily indecomposable or locally connected continuum $X$, the space $\cCK(X)$ is contractible\/ {\rm(\cite[Theorem 20.3, 20.14]{IllanesNadler})}.

If the continuum $X$ is non-degenerate, then $\dim\cCK(X)\ge2$ {\rm(\cite[Theorem 22.18]{IllanesNadler})}.
\end{thm}

\begin{examp}
\begin{enumerate}[leftmargin=*,itemsep=0.3cm]
\item For the standard interval $I$ and the circle $S^1$, the spaces $\cCK(I)$ and $\cCK(S^1)$ are homeomorphic to two-dimensional disk $B^2$ (\cite[ sections 5.1 and 5.2]{IllanesNadler}).
\item The spaces $\cCK_{GH}^c(I)$, $\cCK_{GH}(I^1)$, and $\cCK_{GH}(S^1)$ are homeomorphic to an interval.
\item The space $\cCK_{GH}^c(S^1)$ is homeomorphic to the union of a half-interval and an isolated point. Note that the point $\{S^1\}$ corresponding to the entire circle is at distance $1$ in the $d_{GH}^c$ metric from every proper subset of the circle and is therefore isolated in $\cCK_{GH}^c(S^1)$.
\end{enumerate}
\end{examp}

A continuum $X$ is called \emph{hereditarily indecomposable} if each of its subcontinuum $Y$ cannot be represented as the union of two proper (non-empty and distinct from $Y$) subcontinua.

Recall that we call a topological space $Y$ incomparable with a topological space $X$ if the only continuous mappings from $X$ to $Y$ are mappings to a point.

\begin{prop}\label{prop:LinConnAndContinuum}
Let $X$ be a path-connected space, and $Y$ a hereditarily indecomposable continuum. Then $X$ is incomparable with $Y$. In particular, $2d_{GH}^c(X,Y)\ge\diam X$.
\end{prop}

\begin{proof}
Let $f\:X\to Y$ be a continuous mapping such that $f(X)$ contains at least two distinct points $y_0=f(x_0)\ne f(x_1)=y_1$. Since $X$ is path-connected, there exists a continuous mapping $\v\:[0,1]\to X$ such that $\v(0)=x_0$ and $\v(1)=x_1$. Set $t_0=\sup\bigl[(f\circ \v )^{-1}(x_0)\bigr]=\max\bigl[(f\circ\v )^{-1}(x_0)\bigr]$. Clearly, $\v(t_0)=x_0$. We put $t_1=\inf\bigl[(f\circ\v)^{-1}(x_1)\bigr]\cap[t_0,1]=\min\bigl[(f\circ\v)^{-1}(x_1)\bigr]\cap[t_0,1]$. For an arbitrary $t_0<\t<t_1$, the continuum $K=(f\circ\v)\bigl([t_0,t_1]\bigr)\ss Y$ is represented as the union of two proper subcontinua $y_0\in K_0=(f\circ\v)\bigl([t_0,\t]\bigr)\not\ni y_1$ and $y_0\not\in K_1=(f\circ\v)\bigl([\t,t_1]\bigr)\ni y_1$. The resulting representation contradicts the indecomposability of $Y$, i.e., it means the triviality of any continuous mapping $f\:X\to Y$.
\end{proof}

A \emph{chain\/} in a topological space is a finite sequence of its subsets in which only successive elements intersect (they are called \emph{links\/}). If for some $\e>0$ each link in a chain lying in a metric space has diameter at most $\e$, then such a chain is called an \emph{$\e$-chain}. If for every $\e>0$ the continuum is covered by an $\e$-chain, then such a continuum is called \emph{arc-shaped}. A hereditarily indecomposable arc-shaped continuum is called a \emph{pseudoarc}. An example of a non-degenerate pseudoarc that is a subset of the plane was given by Knaster in~\cite{Knaster}. Moise~\cite{Moise} showed that every proper subcontinuum of a pseudoarc is homeomorphic to the pseudoarc itself. Bing~\cite{BingUnique} proved that all pseudo-arcs are homeomorphic to each other. Bing also found in~\cite{B} that almost all continua in the Euclidean space $\R^n$, $n\ge2$, are pseudo-arcs. More formally, a subset $Y$ of a topological space $X$ is called a \emph{$G_\dl $-set} if $Y$ is equal to at most a countable intersection of open subsets of $X$. We say that \emph{most elements of a complete metric space $X$ are elements of $Y\ss X$} if $Y$ is a dense $G_\dl$-subset of $X$. Recall that the space $\cCK(\R^n)$ of all continua in $\R^n$ is equipped with the Hausdorff metric and the corresponding metric topology.

\begin{thm}[Bing~\cite{B}]\label{thm:Most}
Most continua in $\R^n$, $n\ge2$, i.e., most points from $\cCK(\R^n)$ are pseudo-arcs.
\end{thm}

It follows that the Gromov--Hausdorff distance from any subcontinuum in $\R^n$, $n\ge2$, to the set of all pseudo-arcs is zero, but the continuous Gromov--Hausdorff distance from any path-connected subcontinuum to the set of all pseudo-arcs is equal to the diameter of this subcontinuum.

Let $X$ be a continuum. The distances $d_{GH}^c$ and $d_{GH}$ are pseudometrics on the set $\cCK(X)$. The quotient spaces $\cCK(X)$ by the pseudometrics $d_{GH}^c$ and $d_{GH}$ are denoted by $\cCK_{GH}^c(X)$ and $\cCK_{GH}(X)$, respectively. The projections $p_{GH}^c\:\cCK_{GH}^c(X)\to \cCK_{GH}(X)$ and $p_H\:\cCK_H(X)\to \cCK_{GH}(X)$ are non-expanding mappings, so they are continuous.

\begin{prop}
The space $\cCK_{GH}(X)$ is a path-connected continuum.
\end{prop}

\begin{proof}
By Theorem~\ref{thm:cCKlinearConnectedCont} the space $\cCK(X)$ is a linearly connected continuum.
Now the result follows from the fact that a continuous mapping preserves both compactness and path connectivity.
\end{proof}

H. Cook in~\cite{C67} constructed an example of a hereditarily indecomposable continuum $M_1$ such that any two of its different non-degenerate subcontinua are incomparable.

\begin{thm}[\cite{C67}]
The continuum $M_1$ is one-dimensional in the sense of Lebesgue dimension, therefore $M_1$ is embedded in $\R^3$, but no one non-degenerate subcontinuum of it is embeddable in the plane.
\end{thm}

Let us apply Corollary~\ref{cor:2=}.

\begin{cor}\label{cor:GHCont-intrinsic}
Let $K$ and $H$ be arbitrary distinct subcontinua in the Cook continuum $M_1$, then $2d_{GH}^c(K,H)=\max\{\diam H,\diam K\}$. In particular, the mapping $p_c^{-1}$ is discontinuous on the entire $\cCK_{GH}(M_1)$, except for the one-point space.
\end{cor}

\begin{rk}
According to Item~(\ref{prop:GHelementProps:4}) of Proposition~\ref{prop:GHelementProps}, the value $d_{GH}^c(K,H)$ is the maximum possible value of the $d_{GH}^c$-distance between $H$ and $K$.
\end{rk}

\section{Continuous GH-Distance is Intrinsic}

Let $X$ and $Y$ be arbitrary metric spaces such that $2d^c_{GH}(X,Y)<\infty$. Choose an arbitrary $\e>0$, then there exist continuous $X\torllim^f_g Y$ such that $\dis R_{f,g}<2d^c_{GH}(X,Y)+2\e$. For each $t\in(0,1)$, define a metric $d_t$ on $R:=R_{f,g}$ as follows:
$$
d_t\bigl((x,y),(x',y')\bigr)=(1-t)|xx'|+t\,|yy'|,
$$
and we denote the resulting metric space by $R_t$. We extend $R_t$ by setting $R_0=X$ and $R_1=Y$.

\begin{thm}\label{thm:GHCont-intrinsic}
In the notation introduced above, we have $d^c_{GH}(R_t,R_s)\le |ts|\,\dis R$. In particular, the mapping $\g \:[0,1]\to \GH^c$ is a continuous curve whose length $|\g |$ satisfies the inequality $|\g |<d^c_{GH}(X,Y)+\e$, and the continuous Gromov--Hausdorff distance is intrinsic.
\end{thm}

\begin{proof}
Let $0<t,\,s<1$, and $R'\in \cR(R_t,R_s)$ be the correspondence generated by the identity mapping. Then
\begin{multline*}
\dis R'=\sup_{(x,y),(x',y')\in R}\Bigl|d_t\bigl((x,y),(x',y')\bigr)-d_s\bigl((x,y),(x',y')\bigr)\Bigr|=\\ =
\sup_{(x,y),(x',y')\in R}|ts|\,\bigl||xx'|-|yy'|\bigr|=|ts|\,\dis R.
\end{multline*}
Further, for $t=0$, $X=R_t$ and $R_s$, $0<s<1$, we choose the correspondence $R_{f',g'}$  as $R'\in \cR(R_t,R_s)$, where $f'(x)=\bigl(x,f(x)\bigr)$ for all $x\in X$, and $g'\bigl((x,y)\bigr)=x$. The mapping $f'$ is continuous, since it is the restriction to $X\times f(X)$ of the mapping $x\mapsto\bigl(x,f(x)\bigr)$ from $X$ to $X\times Y$ with continuous coordinate mappings. The mapping $g'$ is also continuous as the restriction of the projection $X\times Y\to X$. Further,
\begin{multline*}
\dis f'=\sup_{x,x'\in X}\Bigl||xx'|-(1-s)|xx'|-s\,\bigl|f(x)f(x')\bigr|\Bigr|=\\ =
\sup_{x,x'\in X}s\,\Bigl||xx'|-\bigl|f(x)f(x')\bigr|\Bigr|=
|ts|\,\dis f\le |ts|\,\dis R;
\end{multline*}
$$
\dis g'=\sup_{(x,y),(x',y')\in R}\bigl||xx'|-(1-s)|xx'|-s\,|yy'|\bigr|=
\sup_{(x,y),(x',y')\in R}s\,\bigl||xx'|-|yy'|\bigr|=|ts|\,\dis R;
$$
\begin{multline*}
\codis (f',g')=\sup_{x\in X,\,(x',y')\in R}\biggl|\Bigl|x\,g'\bigl((x',y')\bigr)\Bigr|-
\bigl|(x',y')\,f'(x)\bigr|\biggr|=\\ =
\sup_{x\in X,\,(x',y')\in R}\Bigl||xx'|-(1-s)|xx'|-s\,\bigl|y'\,f(x)\bigr|\Bigr|=\\ =
\sup_{(x,f(x)),\,(x',y')\in R}s\,\Bigl||xx'|-\bigl|y'\,f(x)\bigr|\Bigr|\le|ts|\,\dis R.
\end{multline*}
Thus, we have shown that $\dis R'\le |ts|\,\dis R$ also in the case $t=0$ and $0<s<1$. The case $0<t<1$ and $s=1$ is analyzed similarly. Finally, for $t=0$ and $s=1$, we set $R'=R$, so that here $\dis R'=|ts|\,\dis R$. Thus, for all $0\le t,\,s\le 1$, we have
$$
d^c_{GH}(R_t,R_s)\le\frac12|ts|\,\dis R,
$$
from which it immediately follows that the mapping $\g\:t\mapsto R_t$ is continuous with respect to $d^c_{GH}$, and the length $|\g|$ of the curve $\g$ does not exceed $\frac12\dis R<d^c_{GH}(X,Y)+\e$. Since $\e>0$ can be chosen arbitrarily small, we conclude that the distance $d^c_{GH}$ is intrinsic.
\end{proof}

As in the standard theory, the correspondence $R_{f,g}\in \cR(X,Y)$ is called \emph{optimal} if $2d^c(X,Y)=\dis R_{f,g}$.

\begin{cor}
If the correspondence $R=R_{f,g}$ is optimal, then the curve $R_t$ constructed above is the shortest geodesic, the length of which is equal to the distance between its ends.
\end{cor}

\section{Incompleteness}

\begin{prop}\label{prop:GHCNotComplete}
The continuous Gromov--Hausdorff distance restricted to the space $\cM^c$ of metric compacts is not complete.
\end{prop}

\begin{proof}
Consider the sequence $X_n=\{i/n\}_{i=0}^n$ of subsets of the interval $[0,1]$ with induced distance. Since the continuous and ordinary GH-distances coincide on finite metric spaces, this sequence is fundamental in $\cM^c$. Suppose that $X_n\tolim^{d^c_{GH}}X$, then $X_n\tolim^{d_{GH}}X$, and since $\mathcal{M}$ is a metric space, the limit is uniquely defined, hence $X=[0,1]$. However, by Proposition~\ref{prop:ConnectDiscrete} we have $d^c(X_n,X)=1/2$, so the sequence $X_n$ does not converge to $X$ in $\cM^c$.
\end{proof}

The validity of Proposition~\ref{prop:GHCNotComplete} is due to the fact that we have restricted the class of metric spaces (compacta) in which we seek the limit of the given fundamental sequence. However, in the class of all metric spaces, this sequence has a limit, namely, it is each zero-dimensional dense subset of the interval. Let us construct a fundamental sequence for which no metric space is its limit w.r.t\. the continuous Gromov-Hausdorff metric.

We will need additional technical result.

\begin{lem}\label{lemmadiametr}
If for a subset $A\ss X$ and a mapping $f\:X\to Y$ at least one of the sets $A$ or $f(A)$ has a finite diameter, then it holds $\dis f\ge\bigl|\diam f(A)-\diam A\bigr|$.
\end{lem}

\begin{proof}
Suppose that only one of the sets $A$ and $f(A)$ has a finite diameter. Then, in the set of finite diameter, each pair of points is at a bounded distance, while in the second, the corresponding points (from the image or preimage of $f$) can be chosen arbitrarily far apart, which proves $\dis f=\infty$, and the inequality holds.

Now let both $A$ and $f(A)$ have finite diameters. We fix $\e>0$.

Let us take points $a,a'\in A$ such that $|aa'|>\diam A-\e $. Then
$$
\dis f\ge |aa'|-\bigl|f(a)f(a')\bigr|>\diam A-\e -\bigl|f(a)f(a')\bigr|\ge\diam A-\e-\diam f(A).
$$
Due to the arbitrariness of the number $\e>0$, we obtain $\dis f\ge\diam A-\diam f(A)$.

Next, we take points $a,a'\in A$ such that $\bigl|f(a)f(a')\bigr|>\diam f(A)-\e$. Then
$$
\dis f\ge\bigl|f(a)f(a')\bigr|-|aa'|>\diam f(A)-\e-|aa'|\ge \diam f(A)-\e-\diam A.
$$
Since the number $\e>0$ is arbitrary, we obtain $\dis f\ge\diam f(A)-\diam A$, which completes the proof.
\end{proof}

The resulting inequality is a strengthening of Proposition~\ref{prop:GHelementProps} Property~(\ref{prop:GHelementProps:5}).

\begin{lem}\label{lemmacodisdistan}
For any mappings $f\:X\to Y$, $g\:Y\to X$, and any point $y_0\in Y$, we have $\codis(f,g)\ge\bigl|y_0\,f(X)\bigr|$.
\end{lem}

\begin{proof}
Take a point $x=g(y_0)\in X$. Then
$$
\codis(f,g)\ge\bigl|f(x)\,y_0\bigr|-\bigl|x\,g(y_0)\bigr|=\bigl|f(x)\,y_0\bigr|\ge\bigl|y_0\,f(X)\bigr|,
$$
which is what was required.
\end{proof}

\begin{examp}\label{examp:noncomplete}
In the plane with coordinates $(x,y)$ we choose the following intervals: $J=\{(0,y):-1\le y\le 1\}$, $I=\{(x,0):0\le x\le 1\}$, and $I_n=\{(x,0):\frac1n\le x\le 1\}$, $n\in \N$. Consider the triode $X=J\cup I$ and the sequence of spaces $X_n=J\cup I_n$, $n\in \N$. On the triode $X=J\cup I$ we take the intrinsic metric, and on its subsets $X_n$ the metric induced by this intrinsic metric.
\end{examp}

\begin{thm}\label{thm:noncomplete}
The sequence of spaces $X_n=J\cup I_n$, $n\in \N$ is fundamental w.r.t\. the metric $d^c_{GH}$, but no metric space is its limit in this metric.
\end{thm}

\begin{proof}
It is easy to see that $d_{GH}^c(X_m,X_n)\le\frac{m-n}{mn}<\frac1n$ for $m\ge n$. Therefore, the sequence of compact metric spaces $X_n$ is fundamental in the metric $d_{GH}^c$.

One can easily verify that $d_{GH}(X,X_n)\le d_{H}(X,X_n)=\frac1{2n}$. Therefore, the triode $X$ is the limit of the sequence $X_n$ in the metric $d_{GH}$.

Let this sequence have a limit $X_\infty $ in the metric $d^c_{GH}$.

From the inequality $d_{GH}\le d_{GH}^c$ it follows that the space $X_\infty $ is also the limit of the sequence $X_n$ in the metric $d_{GH}$. Therefore, $d_{GH}(X,X_\infty)=0$.

According to Theorem~\ref{thm:equivcompGHzero}, we can assume that the space $X_\infty$ belongs to the triode $X$ and is dense there.

Let us show that $\bigl\{(0,-1)\bigr\}\cup\bigl\{(0,1)\bigr\}\cup\bigl\{(1,0)\bigr\}\cup X_\infty=X$. The inclusion of the left set into the right set is obvious.

Take an arbitrary $(x_0,y_0)\in X\sm\Bigl(\bigl\{(0,-1)\bigr\}\cup\bigl\{(0,1)\bigr\}\cup\bigl\{(1,0)\bigr\}\cup X_\infty\Bigr)$.

a. Let $(x_0,y_0)\ne(0,0)$. Without loss of generality we can assume that $(x_0,y_0)=(r,0)$, $r>0$. Then one of the connected components $K$ of the space $X\sm\bigl\{(x_0,y_0)\bigr\}$ has diameter $\diam K=1-r<1$. We fix an index $n$ such that $\diam I_n=1-1/n>\diam K$.

Let $m\ge n$ and the mapping $f\:X_m\to X_\infty$ be continuous. The point $(r,0)$ partitions the space $X_\infty$, so any connected set intersecting $K$ lies entirely in $K$.

If $f(J)\cap K\ne\0$, then $f(J)\ss K$ and $\dis f\ge2-(1-r)=1+r$ according to Lemma~\ref{lemmadiametr}.

If $f(I_m)\cap K\ne \0 $, then $f(I_m)\ss K$ and $\dis f\ge 1-\frac1m-(1-r)=r-\frac1m>r-\frac1n$ according to Lemma~\ref{lemmadiametr}.

If $f(X_m)\cap K=\0 $, then $\codis (f,g)\ge 1-r$ according to Lemma~\ref{lemmacodisdistan} for any mapping $g\:X_m\to X_\infty$.

b. Let $(x,y)=(0,0)$. Then the space $X\sm\bigl\{(x,y)\bigr\}=K_1\sqcup K_2\sqcup K_3$ consists of three connected components of diameter $1$. Hence, for every continuous mapping $f\:X_m\to X_\infty\ss X$, there is a connected component $K_i$ such that $f(J)\ss K_i$. Therefore, $\dis(f)\ge\diam J-\diam K_i=1$ by Lemma~\ref{lemmadiametr}.

In account, we proved that the space $X_\infty$ is connected. This means that for every continuous mapping $g\:X_\infty\to X_m$ of a connected space into a space with two connected components, the latter has a connected component $K$ such that $g(X_\infty)\cap K=\0 $. By Lemma~\ref{lemmacodisdistan}, $\codis (f,g)\ge1-\frac1m$ for any mapping $g\:X_m\to X_\infty$.

Therefore, for any metric space $X_\infty$, the sequence of distances $d_{GH}^c(X_m,X_\infty)$ cannot tend to zero.
\end{proof}

\section{Addition: An Analog of Topology on Proper Classes}\label{sec:add}
The theory we develop in this article deals with all nonempty metric spaces, considered up to isometry. Since a metric can be introduced on every set, for example, by setting all distances between distinct points equal to $1$, the family of all metric spaces is not a set. To work with such families, we will use von Neumann-Bernays-G\"odel set theory~\cite{Mendelson, TBanach}, and we will always assume that the axiom of choice holds. This theory is usually abbreviated NBGC. Recall that all objects in this theory are called \emph{classes\/} and they come in two types: \emph{sets\/} are the classes that are elements of other classes, and \emph{proper classes}, which are not elements of any other classes. Here are some examples of proper classes that are important to us:
\begin{itemize}
\item the class $\cV$ of all sets;
\item the class $\ORD$ of all ordinals;
\item the class $\CARD$ of all cardinals;
\item the class $\TOP$ of all topologies on all sets in $\cV$;
\item the class $\GH$ of all metric spaces, considered up to isometry.
\end{itemize}

Many standard operations are defined for classes, such as intersection, complement, product, mapping, and others.

We will use the following terminology: a \emph{generalized pseudometric on a set $X$} is any mapping $\r\:X\x X\to[0,\infty]$ satisfying the conditions $\r(x,y)=\r(y,x)$ for all $x,y\in X$ (symmetry), $\r(x,x)=0$ for all $x\in X$, and $\r(x,z)\le\r(x,y)+\r(y,z)$ for all $x,y,z\in X$ (triangle inequality). If $\r(x,y)=0$ exactly when $x=y$, then the word ``pseudometric'' is replaced by the word \emph{metric}. If $\r(x,y)<\infty$ for all $x,y\in X$, then the word ``generalized'' is dropped. Since the product and mapping are defined for all classes, the concepts just described carry over word for word to arbitrary classes. Below, we recall the definitions of the Gromov--Hausdorff distance and the continuous Gromov--Hausdorff distance, which, as will be noted, are generalized pseudometrics on the proper class $\GH$.

An interesting feature of proper classes is the impossibility of literally transferring the concept of topology to them: indeed, the entire set on which the topology is defined is an element of the topology, so if we consider a proper class instead of the set, it cannot be an element of the topology. To circumvent this problem and introduce an analogue of topology, it was proposed in~\cite{ITBorzov} to consider \emph{classes $\gC$ filtered by sets}. The latter means that for every cardinal number $n$, the subclass $\gC_n=\{x\in\gC:\#x\le n\}$, where $\#x$ denotes the cardinality of the set $x$, is a set. Examples of such classes include any sets, as well as the proper classes $\ORD$, $\CARD$, and $\GH$. The classes $\cV$ and $\TOP$ are not such.

Let $\gC$ be a set-filtered class. Consider a mapping $\t\:\CARD\to\TOP$ such that
\begin{itemize}
\item $\t_n:=\t(n)$ is a topology on $\gC_n$ for every $n\in\CARD$,
\item for every $m,n\in\CARD$, $m\le n$, the topology $\t_m$ is induced from $\t_n$.
\end{itemize}
Note that if $\gC$ is a set of cardinality $n$, then $\t_n$ is the usual topology on $\gC=\gC_n$, and for every $m\ge n$ we have $\gC_m=\gC_n$, and also $\t_m=\t_n$. For $m\le n$, the topology $\t_m$ on $\gC_m$ is induced from $\t_n$ in the usual way. This allows us to call the mapping $\t$ a \emph{topology\/} even in the case of proper classes. A class $\gC$, filtered by sets, for which the topology $\t$ is defined, will be called a \emph{topological class}.

\begin{rk}
A topology in the above sense can also be defined for any proper class. To do this, it is sufficient to modify the concept of a filtration by sets, without tying it to the cardinality of the elements in the class. It is well known that in NBGC there is a bijection between any proper classes. Therefore, if $\gC$ is a topological class and $\gC'$ is an arbitrary class, say $\cV$, then using the bijection $\v\:\gC\to\gC'$, we can transfer the filtration by sets $\gC_n$ to $\gC'$, setting $\gC'_n=\v(\gC_n)$ and performing all the topological constructions described above. However, for our purposes, the filtration defined by cardinality is sufficient.
\end{rk}

An important special case is the metric topology defined on the class $\gC$, filtered by sets. Let $\r\:\gC\x\gC\to[0,\infty]$ be a generalized pseudometric. Then, for each $n\in\CARD$, there is a corresponding metric topology $\t_n$ on $\gC_n$. The corresponding mapping $\t\:\CARD\to\TOP$, $\t\:n\mapsto\t_n$, is called the \emph{metric topology}, and the space $\gC$ endowed with this topology is called the \emph{metric class}.

Note that for every $x\in\gC$ and $r\in(0,\infty]$, there is an \emph{open ball\/} $U_r(x)=\{y\in\gC:\rho(x,y)<r\}$, which is, generally speaking, a proper subclass of $\gC$. It is easy to see that if $m\in\CARD$ is not less than $\#x$, then the set $U_r(x)\cap\gC_m$ is an open ball of radius $r$ in $\gC_m$.

\begin{prop}\label{prop:SmallDisk}
For any $x\in\gC$, $r>0$, and $m\in\CARD$, the set $U_r(x)\cap\gC_m$ is open in the $\t_m$ topology.
\end{prop}

\begin{proof}
Choose an arbitrary $y\in U_r(x)\cap\gC_m$, then $|xy|<r$. There exists $\dl>0$ such that $|xy|<r-\dl$, so $U_\dl(y)\ss U_r(x)$, but $U_\dl(y)\cap\gC_m$ is an open ball of radius $\dl$ in $\gC_m$, and $U_\dl(y)\cap\gC_m\ss U_r(x)\cap\gC_m$, therefore $U_r(x)\cap\gC_m\in\t_m$.
\end{proof}

If $f\:\gA\to\gB$ is a mapping between two topological classes, then its continuity is defined naturally. Namely, it is known that the image of every set is also a set, so for every cardinal $n$ there exists a cardinal $m$ such that $f(\gA_n)\ss\gB_m$. The smallest of such cardinals $m$ will be denoted by $n_f$. Restricting $f$ to the mapping from $\gA_n$ to $\gB_{n_f}$, we obtain the usual mapping of topological spaces. It is easy to see that if we take a cardinal number $m\ge n_f$ instead of $n_f$, then the restriction of $f$ to $\gA_n$ and $\gB_m$ is continuous (at a point or in the whole) if and only if its restriction to $\gA_n$ and $\gB_{n_f}$ is continuous. Thus, the \emph{continuity\/} of a mapping $f$ at a point or in the whole is defined as the continuity of all its restrictions $f\:\gA_n\to\gB_{n_f}$ (we will continue to denote the restriction by the same letter as the original mapping).

\begin{thm}\label{thm:ContMapMetricClasses}
Let $f\:\gA\to\gB$ be a mapping of metric classes. Then $f$ is continuous at a point $x\in\gA$ if and only if for every $\e>0$, there exists a $\dl>0$ such that $f\bigl(U_\dl(x)\bigr)\ss U_\e\bigl(f(x)\bigr)$. The mapping $f$ is continuous in the whole if and only if it is continuous at every point.
\end{thm}

\begin{proof}
First, let for any $\e>0$ there exists $\dl>0$ such that $f\bigl(U_\dl(x)\bigr)\ss U_\e\bigl(f(x)\bigr)$. We show that $f$ is continuous at $x$. Choose an arbitrary $n$, an arbitrary $\e>0$, and a corresponding $\dl>0$ satisfying the above property. Since $U_\dl(x)\cap\gA_n$ and $U_\e\bigl(f(x)\bigr)\cap\gB_{n_f}$ are open balls in $\gA_n$ and $\gB_{n_f}$, respectively, and by assumption the image of the first is contained in the second, the restriction $f\:\gA_n\to\gB_{n_f}$ is continuous at $x$.

We now prove the converse by contradiction. Namely, suppose that for some $f$ there exist $x\in\gA$ and $\e>0$ such that $f\bigl(U_\dl(x)\bigr)\ss U_\e\bigl(f(x)\bigr)$ does not hold for any $\dl>0$. The latter means that there exists a sequence $y_k\in\gA$ for which $|x\,y_k|<1/k$, but $\bigl|f(x)f(y_k)\bigr|\ge\e$. Let $n$ be the cardinality of $x$, and $n_k$ the cardinality of $y_k$. Since the restriction $f\:\gA_n\to\gB_{n_f}$ is continuous, there exists $s>0$ such that $f$ maps $U_s(x)\cap\gA_n$ into $U_\e\bigl(f(x)\bigr)\cap\gB_{n_f}\ss U_\e\bigl(f(x)\bigr)$. Thus, we can immediately assume that all $n_k$ are at least $n$.

Since a countable union of sets is still a set, there exists $p\in\CARD$ such that $p\ge n_k$ for all $k$. Note that $x\in\gA_p$, $f(x)\in\gB_{p_f}$, and the restriction $f\:\gA_p\to\gB_{p_f}$ are continuous, so there exists $s>0$ such that the $f$-image of the ball $U_s(x)\cap\gA_p$ of radius $s$, open in $\gA_p$, is contained in the ball $U_\e\bigl(f(x)\bigr)\cap\gB_{p_f}\ss U_\e\bigl(f(x)\bigr)$ of radius $\e$, open in $\gB_{p_f}$. But then for $1/k<s$, we have $y_k\in U_s(x)\cap\gA_p$, and hence $\bigl|f(x)f(y_k)\bigr|<\e$, a contradiction.
\end{proof}

\begin{cor}\label{cor:ContLipschMetrCl}
Let $f\:\gA\to\gB$ be a Lipschitz mapping of metric classes, in particular, non-expanding. Then $f$ is continuous.
\end{cor}

A mapping $f\:\gA\to\gB$ of metric classes is called \emph{open at $x\in\gA$} if for every $r>0$, there exists $s>0$ such that $f\bigl(U_r(x)\bigr)\sp U_s\bigl(f(x)\bigr)$. The following result is immediately implied by the theorem.

\begin{cor}\label{cor:ContInverseMetrCl}
Let $f\:\gA\to\gB$ be a bijective mapping of metric classes. Then $f^{-1}$ is continuous at $z=f(x)\in\gB$ if and only if $f$ is open at $x$.
\end{cor}

\markright{Bibliography}


\begin{thebibliography}{20}
\bibitem{Rieffel} \href{https://arxiv.org/pdf/math/0011063}{Rieffel, M.\,A. 2003, ``Gromov-Hausdorff Distance for Quantum Metric Spaces'', \textit{ArXiv e-prints}, arXiv:math/0011063[math.OA].}

\bibitem{LimMemoliSmith} Lim, S., Memoli, F. \& Smith, Z. 2023, ``The Gromov–Hausdorff distance between spheres'', \textit{Geometry \& Topology}, vol. 27, no. 9, pp. 3733--3800.

\bibitem{LeeMorales}  Lee, J. \& Morales, C.\,A. 2022, ``Gromov-Hausdorff Stability of Dynamical Systems and Applications to PDEs'', Birkh\"auser/Springer.

\bibitem{Mendelson} Mendelson, E. 1984, ``Introduction to Mathematical Logic'', M.: Nauka.

\bibitem{TBanach} \href{https://arxiv.org/pdf/2006.01613}{Banach, T. 2023, ``Classical set theory: theory of sets and classes'', \textit{ArXiv e-prints}, arXiv:2006.01613[math.LO].}

\bibitem{Gromov1981} Gromov, M. 1981, ``Structures m\'etriques pour les vari\'et\'es riemanniennes'', Edited by Lafontaine and Pierre Pansu.

\bibitem{Gromov1999} Gromov, M. 1999, ``Metric structures for Riemannian and non-Riemannian spaces'', Birkh\"auser.

\bibitem{BurBurIva01} Burago, D., Burago, Yu. \& Ivanov, S. 2001, ``A Course in Metric Geometry'', Providence.

\bibitem{B} Bing, R.\,H. 1948, ``A homogeneous indecomposable plane continuum'', \textit{Duke Math. J.}, vol. 15, pp. 729--742.

\bibitem{BT21} \href{https://arxiv.org/pdf/2110.06101}{Bogatyy, S.\,A. \& Tuzhilin, A.\,A. 2021, ``Gromov--Hausdorff class: its completeness and cloud geometry'', \textit{ArXiv e-prints}, arXiv:2110.06101[math.MG].}

\bibitem{BT23} Bogatyy, S.\,A. \& Tuzhilin, A.\,A. 2023, ``Action of similarity transformation on families of metric spaces'', \textit{Itogi Nauki i Tekhn.}, vol. 223, pp. 3--13.

\bibitem{BBRT25} Bogataya, S.\,I., Bogatyy, S.\,A., Redkozubov, V.\,V. \& Tuzhilin, A.\,A. 2023, ``Clouds in Gromov--Hausdorff Class: their completeness and centers'', \textit{Topology and its Applications}, vol. 329.

\bibitem{Fet}  Fet, A.\,I. 1954, ``A generalization of the Lyusternik-Shnirelman theorem on coverings of spheres and some related theorems'', \textit{DAN}, vol. 95, no. 6, pp. 1149--1151.

\bibitem{VikhrovCheb} Vikhrov, A.\,A. 2025, ``The problem of constructing geodesics in the Gromov--Hausdorff class: an optimal Hausdorff implementation does not always exist'', \textit{Chebyshevski sb.}, vol. 26, no. 2, pp. 49--60.

\bibitem{VikhrovSerb} Vikhrov, A. 2025, ``Geometry of linear and nonlinear geodesics in the proper Gromov--Hausdorff class'', \textit{Matematicki vesnik}, pp. 1--17.

\bibitem{Stone} Stone, A.\,H. 1948, ``Paracompactness and product spaces'', \textit{Bull. Amer. Math. Soc.}, vol. 54, pp. 977--982.

\bibitem{E85} Engelking, R. 1985, ``General Topology'', Warszawa.

\bibitem{Munkres} Munkres, J.\,R., 2000, ``Topology'', 2nd Edition, Prentice Hall, Upper Saddle River.

\bibitem{Dowker} Dowker, C.\,H. 1947, ``Mapping theorems for non-compact spaces'', \textit{Amer. J. Math.}, vol. 69, pp. 200--242.

\bibitem{NR65} Nagami, K. \& Roberts, J.\,H. 1965, ``Metric-dependent dimension functions'', \textit{Proc. Amer. Math. Soc.}, vol. 16, no. 4, pp. 601--604.

\bibitem{IllanesNadler} Illanes, A. \& Nadler, S.\,Jr. 1999, ``Hyperspaces'', Marcel Dekker, New York.

\bibitem{Knaster} Knaster, B. 1922, ``Un continu dont tout sous-continu est indecomposable'', \textit{Fund. Math.}, vol. 3, pp. 247--286.

\bibitem{Moise} Moise, E.\,E. 1948, ``An indecomposable plane continuum which is homeomorphic to each of its nondegenerate subcontinua'', \textit{Trans. Amer. Math. Soc.}, vol. 63, pp. 581--594.

\bibitem{BingUnique} Bing, R.\,H. 1951, ``Concerning hereditarily indecomposable continua'', \textit{Pacific J. Math.}, vol. 1, pp. 43--51.

\bibitem{C67} Cook, H. 1967, ``Continua which admit only the identity mapping onto non-degenerate subcontinua'', \textit{Fundamenta Mathematicae}, vol. 60, pp. 241--249.

\bibitem{ITBorzov} \href{https://arxiv.org/pdf/2009.00458}{Borzov, S.\,I., Ivanov, A.\,O. \& Tuzhilin, A.\,A. 2020, ``Extendability of Metric Segments in Gromov-Hausdorff Distance'', \textit{ArXiv e-prints}, arXiv:2009.00458[math.MG].}

\end{thebibliography}
\end{document}